\documentclass[11pt]{amsart}
\usepackage{times}
\usepackage[english]{babel}
\usepackage[left=3.1cm,right=3.1cm,top=3cm,bottom=3cm]{geometry}
\usepackage{amsmath,amsfonts,amssymb,amsthm,cases,upgreek}
\usepackage{empheq}
\usepackage{stmaryrd}

\usepackage{subfigure}  
\usepackage{dsfont}

\usepackage{color}

\theoremstyle{plain}

\newtheorem{thm}{Theorem}[section]

\newtheorem{lemma}[thm]{Lemma}

\newtheorem{prop}[thm]{Proposition}
\newtheorem{remark}[thm]{Remark}

\newcommand{\R}{\mathbb{R}}

\newcommand{\RD}{\mathbb{R}^{d}}

\newcommand{\lver}{\left\lvert}

\newcommand{\rver}{\right\rvert}

\newcommand{\Hmu}{\mathcal{H}_{\mu}}
\newcommand{\Ub}{\textbf{U}}
\newcommand{\N}{\mathbb{N}}

\DeclareMathOperator{\sech}{sech}

\def \epsilon {\varepsilon}

\numberwithin{equation}{section}

\begin{document}

\title{A splitting method for deep water with bathymetry} 

\author[A. Bouharguane ]{Afaf Bouharguane}
\address{Univ. Bordeaux, IMB, INRIA MEMPHIS, UMR 5251, F-33400 Talence, France.}
\email{afaf.bouharguane@math.u-bordeaux1.fr}

\author[B. Melinand]{Benjamin Melinand}
\address{Univ. Bordeaux, IMB,  UMR 5251, F-33400 Talence, France.}
\email{benjamin.melinand@math.u-bordeaux1.fr}

\begin{abstract}
\noindent In this paper we derive and prove the wellposedness of a deep water model that generalizes the Saut-Xu system for nonflat bottoms. Then, we present a new numerical method based on a splitting approach for studying this system. The advantage of this method is that it does not require any low pass filter to avoid spurious oscillations. We prove a local error estimate and we show that our scheme represents a good approximation of order one in time. Then, we perform some numerical experiments which confirm our theoretical result and we study three physical phenomena : \textcolor{black}{the evolution of water waves over a rough bottom}; the evolution of a KdV soliton when the shallowness parameter increases; the homogenization effect of rapidly varying topographies on water waves.
\end{abstract}

\maketitle

\section{Introduction} 

\subsection{Presentation of the problem}

\noindent Understanding the influence of the topography on water waves is an important issue in oceanography. Many physical phenomena are linked to the variation of the topography : shoaling, rip currents, diffraction, Bragg reflection. Since the direct study on the Euler equations is quite involved, several authors derived and justified asymptotic models according to different small parameters. A usual way to derive asymptotic models is to start from the Zakharov/Craig-Sulem-Sulem formulation \cite{Zakharov,Craig_Sulem_1,Craig_Sulem_2}, which is a good formulation for irrotational water waves, and to expand the Dirichlet-Neumann operator. Then, in the shallow water regime for example, several models were obtained like the Saint-Venant equations or the Green-Naghdi or Boussinesq equations, see \cite{Alvarez_Lannes,Lannes_ww}, \cite{Iguchi_shallow_water} for instance. The present paper addresses the influence of the bathymetry in deep water, in the sense explained below.

\medskip

\noindent In this paper,  $a$ denotes the typical amplitude of the water waves, $L$ the typical length, $H$ the typical height and $a_{bott}$ the typical amplitude of the bathymetry. Then, we introduce three parameters : $\epsilon = \frac{a}{H}$ the nonlinearity parameter, $\mu = \frac{H^{2}}{L^{2}}$ the shallowness parameter and $\beta=\frac{a_{bott}}{H}$ the bathymetric parameter. We recall that assuming $\mu$ small leads to shallow water models. In deep water, which is typically the case when $\mu$ is of order $1$, it is quite common to assume that the steepness parameter $\epsilon \sqrt{\mu} = \frac{a}{L}$ is small. The first asymptotic model with a small steepness assumption was derived by Matsuno in 2D for a flat and non-flat bottom and weakly transverse 3D water waves \cite{matsuno_flat,matsuno_bott}. Then, Choi extended this result in 3D for flat bottom \cite{Choi}. Finally, Bonneton and Lannes  gave a 3D version in the case of a non-flat bottom \cite{bonneton_lannes}. It is 
 important to notice that these models are only formally derived. It is proven in \cite{Alvarez_Lannes} that smooth enough solutions to theses models are close to the solutions of the water waves equations but, to the best of our knowledge, the wellposedness of the Mastuno equations, even in the case of a flat bottom, is still an open problem. This system could be illposed (see Ambrose, Bona and Nicholls \cite{bona_ill}). To avoid this difficulty, Saut and Xu (\cite{saut_xu}) developed an \textcolor{black}{equivalent system} to the Matsuno system which is consistent with the water waves problem and with the same accuracy. Then, they proved that this new system is wellposed. However, this model is for a flat bottom. In this paper, we derive (see Section \ref{saut_xu_system_derivation}), use, and prove the wellposedness of a generalization of the Saut-Xu system with a non-flat bottom which is the following system

\begin{equation}\label{saut-xu-deep}
\left\{
\begin{aligned}
& \hspace{-0.05cm} \partial_{t} \zeta \hspace{-0.1cm} - \hspace{-0.1cm} \Hmu v \hspace{-0.1cm} + \hspace{-0.1cm} \epsilon \sqrt{\mu} \left( \frac{1}{2} v \partial_{x} \zeta  \hspace{-0.1cm} + \hspace{-0.1cm} \frac{1}{2} \Hmu \hspace{-0.05cm} \left(v \partial_{x} \hspace{-0.05cm} \Hmu \zeta \hspace{-0.05cm} \right) \hspace{-0.1cm} + \hspace{-0.1cm} \Hmu \hspace{-0.05cm} \left( \hspace{-0.05cm} \zeta \partial_{x} \Hmu v \hspace{-0.05cm} \right) \hspace{-0.1cm} + \hspace{-0.1cm} \zeta \partial_{x} v \right) = \beta \sqrt{\mu} \partial_{x}  \hspace{-0.05cm} \left( \hspace{-0.03cm} B_{\mu} v \right)\\
& \hspace{-0.05cm}  \partial_{t} v + \partial_{x} \zeta + \frac{3 \epsilon \sqrt{\mu}}{2} v \partial_{x} v - \frac{\epsilon \sqrt{\mu}}{2} \partial_{x} \zeta \Hmu \partial_{x} \zeta - \frac{\epsilon \sqrt{\mu}}{2} v \Hmu^{2} \partial_{x} v = 0,
\end{aligned}
\right.
\end{equation}

\noindent where (see Subsection \ref{notations} for the notations)

\begin{equation*}
\Hmu = - \frac{\tanh(\sqrt{\mu} D)}{D} \partial_{x} \text{   and   } B_{\mu} = \text{sech}( \sqrt{\mu} D) \left(b \; \text{sech}( \sqrt{\mu} D) \;  \cdot \; \right).
\end{equation*}

\noindent Many authors developed numerical approaches to study the impact of the bottom on water waves, see for instance \cite{Mei_bragg}, \cite{liu_bragg}, \cite{smith}, \cite{guyenne_nicholls_solitary_slopes}, \cite{Cathala}, \cite{nachbin_nonsmoooth}, \cite{hamilton_rapidlybott}, \cite{bonneton_lannes_marche_serre}, \cite{bonneton_lannes_marche_splitting}). However to the best of our knowledge, when one works with deep water, there is no convergence result in the literature. After the original work of Craig and Sulem (\cite{Craig_Sulem_1}) and the paper of Craig et al. (\cite{craig_rough_bott}), Guyenne and Nicholls (\cite{numeric_Guyenne_moving_bott}) developed a numerical method based on a pseudospectral method and a fourth-order Runge-Kutta scheme for the time integration. The linear terms are solved exactly whereas the nonlinear terms are viewed as source terms. Their approach has been developed for the whole water waves equations but we could easily adapt it to our system. However with their scheme, we observe spurious oscillations in the wave profile that lead to instabilities. These errors seem to appear when the nonlinear part is evaluated via the Fourier transform. This is the aliasing phenomenon. Guyenne and Nicholls also observe these oscillations and, to fix it, they apply at every time step a low-pass filter. The scheme that we propose in this paper avoids this low-pass filter.

\medskip 
\medskip

\noindent We present a new numerical method based on a splitting approach for studying nonlinear water waves in the presence of a bottom. We remark that the Saut-Xu system contains a dispersive part and a nonlinear transport part. Thus, the splitting method becomes an interesting alternative to solve the system since this approach is commonly used to split different physical terms, see for instance \cite{RS09}. We also motivate our decomposition by the fact that, due to the pseudodifferential operator, some terms in the dispersive part may be computed efficiently using the fast Fourier transform. The transport part is computed \textcolor{black}{by a Lax-Wendroff method}. Various versions of the splitting method have been developed for instance for the nonlinear Schrodinger, the viscous Burgers equation, Korteweg-de-Vries equations \cite{carles_split,HoldenLR13,Lu08,Sa07,TA84}. Thanks to this splitting, we only use a pseudospectral method for the nonlocal terms (contrary to \cite{Craig_Sulem_1,numeric_Guyenne_moving_bott}), which limits the aliasing phenomenon and allows us to avoid a low-pass filter.

\medskip
\medskip

\noindent We denote by $\Phi^t$ the nonlinear flow associated to the Saut-Xu system \eqref{saut-xu-deep}, $\Phi^t_\mathcal{A}$ and $\Phi^t_{\mathcal{D}}$, respectively, the evolution operator associated with the transport part (see equation \eqref{transport_part}) and with the dispersive part (see \eqref{dispersive_part}). We consider the Lie formula defined by

\begin{equation}
\mathcal{Y}^t =  \Phi^t_{\mathcal{A}} \Phi^t_{\mathcal{D}}.  \\
\end{equation}

\noindent The Saut-Xu system \eqref{saut-xu-deep} is a quasilinear system. This implies derivatives losses in the proof of the convergence. In Theorem \ref{thm_global_error}, we show that the numerical solution converges to the solution of the Saut-Xu system \eqref{saut-xu-deep} in the $H^{N+\frac{1}{2}} \times H^{N}$-norm for initial data in $H^{N+\frac{1}{2}-7} \times H^{N-7}(\R)$, \textcolor{black}{where $N \geq 7.$}

\medskip
\medskip

\noindent Notice that it is not \textcolor{black}{hard} to generalize the present work to the Lie formula $\Phi^t_\mathcal{D} \Phi^t_{\mathcal{A}}$. We also make the choice to prove a convergence result for a Lie splitting but our proof can be adapted to a Strang splitting or a more complex one. Finally, notice that our scheme can be used for other equations. 

\medskip
\medskip

\noindent The paper is organised as follows. In the next section,  we extend the Saut-Xu system by adding a \textcolor{black}{topography term} and we prove a local wellposedness result. We also show that the flow map $\Phi^t$ is uniformly Lipschitzean. In section \ref{proof_splitting}, we  split the problem and we give some estimates on $\Phi^t_\mathcal{A}$ and $\Phi^t_\mathcal{D}$. In Section \ref{mainsection}, we prove a local error estimate and  we show that the Lie method represents a good approximation of order one in time (Theorem \ref{thm_global_error}). Finally, in Section \ref{numsection}, we perform some numerical experiments which confirm our theoretical result and we illustrate three physical phenomena :  \textcolor{black}{the evolution of water waves over a rough bottom}; the evolution of a KdV soliton when the shallowness parameter increases and the homogenization effect of rapidly varying topographies on water waves.

\medskip 

\subsection{Notations and assumptions}\label{notations} 

\begin{itemize}
\item $x$ denotes the horizontal variable and $z$ the vertical variable. In this paper, we only study the 2D case ($x\in \R$). 
\item We assume that 

\begin{equation}
\label{parameters_constraints}
0 \leq \epsilon, \beta \leq 1 \text{,   } \exists \mu_{\max} > \mu_{\min} > 0 \text{, } \mu_{\max} \geq \mu \geq \mu_{\min}.
\end{equation}

\item We denote  $\delta = \max(\epsilon, \beta)$.

\item \textcolor{black}{We denote $\Lambda = (1- \partial_x^2)^{1/2} $ and $H^s(\R) = \left\{ u \in L^2(\R), ||u||_{H^2} = ||\Lambda^s u||_{L^2} < \infty \right\}  $ the usual Sobolev space for $s \geq 0$.  }

\item Let $f \in \mathcal{C}^{0} \left(\R \right)$ and $m \in \mathbb{N}$ such that $ \frac{f}{1+|x|^{m}}\in L^{\infty} \left(\R \right)$. We define the Fourier multiplier $f(D) : H^{m}\left(\R \right) \shortrightarrow L^{2}\left(\R \right)$ as

\begin{equation*}
\forall u \in H^{m}\left(\R \right) \text{, } \widehat{f(D) u}(\xi) = f(\xi) \widehat{u}(\xi).
\end{equation*}

\item $D$ denotes the Fourier multiplier corresponding to $\frac{\partial_{x}}{i}$. 
\item We denote by $C(c_1,c_2,...)$ a generic positive constant, strictly positive, which depends on parameters $c_1,c_2,\cdots$. 
\end{itemize}

\section{The Saut-Xu system}\label{saut_xu_system_derivation}

\noindent In this part, we extend the Saut-Xu system (\cite{saut_xu}) for a non-flat bottom. Then, we give a wellposedness result that generalizes the one of Saut and Xu.

\medskip
\medskip

\noindent The Matsuno system, which is a full dispersion model for deep waters, is an asymptotic model of the water waves equations with an accuracy of order $\mathcal{O} \left(\delta^{2} \right)$. Bonneton and Lannes \cite{bonneton_lannes} formulated it in the following way in the presence of a non flat topography

\begin{equation}\label{matsuno}
\left\{
\begin{array}{l}
\partial_{t} \zeta - \frac{1}{\sqrt{\mu} \nu} \mathcal{H}_{\mu} v + \frac{\epsilon}{\nu} \left( \Hmu \left(\zeta \partial_{x} \Hmu v \right) + \partial_{x} \left( \zeta v \right) \right) =  \frac{\beta}{\nu} \partial_{x} \left(B_{\mu} v \right) \\
 \partial_{t} v + \partial_{x} \zeta + \frac{\epsilon}{\nu} v \partial_{x} v -\epsilon \sqrt{\mu} \partial_{x} \zeta \Hmu \partial_{x} \zeta = 0,
\end{array}
\right.
\end{equation}
where $\zeta = \zeta(t,x)$ is the free surface, $v = v(t,x)$ is the horizontal velocity at the surface, $\nu = \frac{\tanh(\sqrt{\mu})}{\sqrt{\mu}}$, $\Hmu$ and $B_{\mu}$ are Fourier multipliers, 

\begin{equation*}
\Hmu = - \frac{\tanh(\sqrt{\mu} D)}{D} \partial_{x} \text{   and   } B_{\mu} = \text{sech}( \sqrt{\mu} D) \left(b \; \text{sech}( \sqrt{\mu} D) \;  \cdot \; \right),
\end{equation*}

\noindent and \textcolor{black}{$-1+\beta b$ is the topography}. \textcolor{black}{(We erased the fluid part)}
\medskip
\medskip

\noindent In \cite{Alvarez_Lannes}, Alvarez-Samaniego and Lannes show that this model is consistent with the Zakharov/Craig-Sulem-Sulem formulation when $\beta = 0$  and it is not painful to generalize their result to the case when $\beta \neq 0$. In \cite{saut_xu}, Saut and Xu obtained a new model with the same accuracy \textcolor{black}{with the Matsuno system} thanks to a nonlinear change of variables. \textcolor{black}{Notice that this change of variables is inspiblack by \cite{bona_colin_lannes}}. The advantage of this model is that they proved a local wellposedness on large time for this new model. We follow their approach. We define new variables

\begin{equation}\label{change_variable}
\widetilde{v} = v + \frac{\epsilon \sqrt{\mu}}{2} v \Hmu \partial_{x} \zeta \text{   and   } \widetilde{\zeta} = \zeta  - \frac{\epsilon \sqrt{\mu}}{4} v^{2}.
\end{equation}

\noindent Then, up to terms of order $\mathcal{O} \left( \delta^{2} \right)$, $\widetilde{\zeta}$ and $\widetilde{v}$ satisfy (we omit the tildes for the sake of simplicity) 

\small
\begin{equation}\label{saut-xu}
\left\{
\begin{aligned}
& \hspace{-0.15cm} \partial_{t} \zeta \hspace{-0.07cm} + \hspace{-0.07cm} \left(\hspace{-0.07cm} \frac{\epsilon}{\nu} \hspace{-0.07cm} - \hspace{-0.07cm} \frac{\epsilon \sqrt{\mu}}{2} \right) v \partial_{x} \zeta \hspace{-0.07cm} - \hspace{-0.07cm} \frac{1}{\sqrt{\mu} \nu} \Hmu v \hspace{-0.07cm} + \hspace{-0.07cm} \frac{\epsilon}{\nu} \left( \frac{1}{2} \Hmu \left(v \partial_{x} \Hmu \zeta \right) \hspace{-0.07cm} + \hspace{-0.07cm} \Hmu  \hspace{-0.07cm} \left( \hspace{-0.02cm} \zeta  \hspace{-0.02cm} \partial_{x}  \hspace{-0.02cm} \Hmu v \hspace{-0.02cm} \right) \hspace{-0.09cm} + \hspace{-0.05cm} \zeta  \hspace{-0.03cm} \partial_{x}  \hspace{-0.03cm} v  \hspace{-0.08cm}\right) \hspace{-0.1cm} = \hspace{-0.1cm} \frac{\beta}{\nu} \partial_{x} \left(B_{\mu} v \right)\\
& \hspace{-0.15cm} \partial_{t} v + \left(\frac{\epsilon}{\nu} + \frac{\epsilon \sqrt{\mu}}{2} \right) v \partial_{x} v + \partial_{x} \zeta - \frac{\epsilon \sqrt{\mu}}{2} \partial_{x} \zeta \Hmu \partial_{x} \zeta - \frac{\epsilon}{2 \nu} v \Hmu^{2} \partial_{x} v = 0.
\end{aligned}
\right.
\end{equation}
\normalsize

\noindent Since our motivation is the study of water waves in deep water ($\mu$ close to $1$), we assume that $\nu=\frac{1}{\mu}$. \textcolor{black}{Hence we study the following system which is a variable bottom analog of the system of Saut and Xu}

\begin{equation*}
\left\{
\begin{aligned}
& \hspace{-0.05cm} \partial_{t} \zeta \hspace{-0.1cm} - \hspace{-0.1cm} \Hmu v \hspace{-0.1cm} + \hspace{-0.1cm} \epsilon \sqrt{\mu} \left( \frac{1}{2} v \partial_{x} \zeta  \hspace{-0.1cm} + \hspace{-0.1cm} \frac{1}{2} \Hmu \hspace{-0.05cm} \left(v \partial_{x} \hspace{-0.05cm} \Hmu \zeta \hspace{-0.05cm} \right) \hspace{-0.1cm} + \hspace{-0.1cm} \Hmu \hspace{-0.05cm} \left( \hspace{-0.05cm} \zeta \partial_{x} \Hmu v \hspace{-0.05cm} \right) \hspace{-0.1cm} + \hspace{-0.1cm} \zeta \partial_{x} v \right) = \beta \sqrt{\mu} \partial_{x}  \hspace{-0.05cm} \left( \hspace{-0.03cm} B_{\mu} v \right)\\
& \hspace{-0.05cm}  \partial_{t} v + \partial_{x} \zeta + \frac{3 \epsilon \sqrt{\mu}}{2} v \partial_{x} v - \frac{\epsilon \sqrt{\mu}}{2} \partial_{x} \zeta \Hmu \partial_{x} \zeta - \frac{\epsilon \sqrt{\mu}}{2} v \Hmu^{2} \partial_{x} v = 0,
\end{aligned}
\right.
\end{equation*}

\noindent In the following, we denote $\textbf{U} = \left(\zeta, v \right)^{t}$ and we define the energy of the system for $N \in \mathbb{N}$ by

\begin{equation}\label{energy}
\mathcal{E}^{N} \!\! \left(\textbf{U} \right) = \frac{1}{\sqrt{\mu}} \left\lvert \Lambda^{N} \zeta \right\rvert_{2}^{2} + \left\lvert |D|^{\frac{1}{2}} \Lambda^{N} \zeta \right\rvert_{2}^{2} + \left\lvert v \right\rvert_{H^{N}}^{2},
\end{equation}
\textcolor{black}{where $ \Lambda = \sqrt{1 + |D|^2} $ and $D = - i \nabla $}. We also denote by $E^{N}_{\mu}$ the energy space related to this norm. 

\begin{remark}\label{explication_assumption_mu}
\noindent Notice that if $\mu$ satisfies condition \eqref{parameters_constraints}, the energy $\mathcal{E}^{N}$ is equivalent to the $H^{N+\frac{1}{2}} \times H^{N}$-norm.
\end{remark}

\noindent The main result of this section is the following local wellposedness result. 

\begin{thm}\label{existence_saut_xu}
Let $N \geq 2$, \upshape$\textbf{U}_{0} \in H^{N+\frac{1}{2}} \left(\R \right) \times H^{N} \left(\R \right)$ \itshape and \upshape $b \in L^{\infty} \left(\R \right)$\itshape. We assume that $\epsilon, \beta, \mu$ satisfy Condition \eqref{parameters_constraints} and 

\upshape
\begin{equation*}
\lver \textbf{U}_{0} \rver_{H^{N+\frac{1}{2}} \times H^{N}} + \lver b \rver_{L^{\infty}} \leq M.
\end{equation*}
\itshape

\noindent Then, there exists a time $T_{0} = T_{0} \left(M, \frac{1}{\mu_{\min}}, \mu_{\max}  \right)$ independent of $\epsilon$, $\mu$ and $\beta$ and a unique solution \upshape$\textbf{U} \in \mathcal{C} \left(\left[0, \frac{T_{0}}{\delta} \right], E^{N}_{\mu} \right)$ \itshape of the system \eqref{saut-xu-deep} with initial data \upshape$\textbf{U}_{0}$\itshape. Furthermore, we have the following energy estimate, for all $t \in \left[0, \frac{T_{0}}{\delta} \right]$,

\upshape
\begin{equation*}
\mathcal{E}^{N} \!\! \left(\textbf{U}(t, \cdot) \right) \leq e^{\delta C_{0} t} \mathcal{E}^{N} \!\! \left(\textbf{U}_{0} \right),
\end{equation*}
\itshape

\noindent where $C_{0} = C \left(M, \frac{1}{\mu_{\min}}, \mu_{\max}  \right)$.
\end{thm}

\begin{proof}
We refer to Paragraph IV in \cite{saut_xu} for a complete proof and we focus only on the bottom contribution. For $0 \leq \alpha \leq N$, we denote $\textbf{U}^{(\alpha)} = \left(\partial_{x}^{\alpha} \zeta, \partial_{x}^{\alpha} v \right)$. Then, applying $\partial_{x}^{\alpha}$ to System \eqref{saut-xu-deep}, we get

\begin{equation*}
\partial_{t} \textbf{U}^{(\alpha)} + \mathcal{L} \textbf{U}^{(\alpha)} + \frac{\epsilon \sqrt{\mu}}{2}  \mathds{1}_{\left\{ \alpha \neq 0 \right\}} \mathcal{B}[\textbf{U}] \textbf{U}^{(\alpha)} = \beta \sqrt{\mu} \left(\partial_{x} \partial_{x}^{\alpha} \left(B_{\mu} v \right), 0 \right)^{t} + \epsilon \sqrt{\mu} \mathcal{G}^{\alpha},
\end{equation*} 

\noindent where 

\begin{equation*}
\begin{aligned}
&\mathcal{L} = \begin{pmatrix} 0 & - \frac{1}{\sqrt{\mu} \nu} \Hmu \\ \partial_{x} & 0 \end{pmatrix}, \\
&\mathcal{B}[\textbf{U}] = \begin{pmatrix} \Hmu \left( v  \Hmu \partial_{x} \; \cdot \;  \right) + v \partial_{x} & \Hmu \left(  \; \cdot \; \Hmu \partial_{x} \zeta \right) - \partial_{x} \zeta \Hmu^{2} \\ -  \partial_{x} \zeta \Hmu \partial_{x} -  \Hmu \partial_{x} \zeta \partial_{x} & 3 v \partial_{x} - v \Hmu^{2} \partial_{x} \end{pmatrix},
\end{aligned}
\end{equation*}

\noindent and $\mathcal{G}^{\alpha} = (\mathcal{G}^{\alpha}_{1}, \mathcal{G}^{\alpha}_{2})^{t}$ with

\small
\begin{equation*}
\begin{aligned}
&\mathcal{G}^{\alpha}_{1} = \partial_{x}^{\alpha} g(\zeta,v)  -  \frac{1}{2} \underset{1 \leq \gamma \leq \alpha-1}{\sum} \hspace{-0.4cm} C^{\gamma}_{\alpha} \left( \Hmu ( \partial_{x}^{\gamma} v \Hmu \partial_{x}^{1+\alpha-\gamma} \zeta) + \partial_{x}^{\gamma} v \partial_{x}^{1+\alpha-\gamma} \zeta \right) -  \frac{1}{2} \partial_{x} \zeta (\Hmu^{2}+1) \partial_{x}^{\alpha} v\\
&\mathcal{G}^{\alpha}_{2} = \frac{1}{2} \underset{1 \leq \gamma \leq \alpha-1}{\sum} \hspace{-0.3cm} C^{\gamma}_{\alpha} \partial_{x}^{1+\gamma} \zeta \Hmu \partial_{x}^{1+\alpha-\gamma} \zeta + \underset{1 \leq \gamma \leq \alpha}{\sum} C^{\gamma}_{\alpha} \left( -\frac{3}{2} \partial_{x}^{\gamma} v \partial_{x}^{1+\alpha-\gamma} v + \frac{1}{2} \partial_{x}^{\gamma} v \Hmu^{2} \partial_{x}^{1+\alpha-\gamma} v \right)
\end{aligned}
\end{equation*}
\normalsize

\noindent where

\begin{equation*}
g(\zeta,v) = - [\Hmu, \zeta] \Hmu \partial_{x} v - \zeta (\Hmu^{2}+1) \partial_{x} v.
\end{equation*}

\noindent Then we can show, as in Paragraph IV. B in \cite{saut_xu} \textcolor{black}{(see the paragraph called Estimate on $\mathcal{G}^{\alpha}$) that}

\begin{equation}\label{estim_1}
\lver \mathcal{G}^{\alpha} \rver_{2} + \lver \lver D \rver^{\frac{1}{2}} \mathcal{G}^{\alpha} \rver_{2} \leq C \left(\frac{1}{\mu_{\min}} \right) \mathcal{E}^{N} \! \left( \textbf{U} \right).
\end{equation}

\noindent \textcolor{black}{Like Saut and Xu we define a symmetrizer for $\mathcal{L} + \mathcal{B}[\textbf{U}]$}

\begin{equation}\label{symmetrizer}
\mathcal{S} = \begin{pmatrix} \frac{D}{\tanh(\sqrt{\mu} D)} & 0 \\ 0 & 1 \end{pmatrix}.
\end{equation}

\noindent Notice that $\sqrt{\left(\mathcal{S} \;  \cdot , \cdot \right)}$ is a norm equivalent to $\sqrt{\mathcal{E}^{0}}$. \textcolor{black}{Then, as in Paragraph IV. B in \cite{saut_xu} (see the paragraph called Estimate on \textit{II}), we get }

\color{black}
\begin{equation}\label{estim_2}
\left(\left( \mathcal{L} + \frac{\epsilon \sqrt{\mu}}{2}  \mathds{1}_{\left\{ \alpha \neq 0 \right\}} \mathcal{B}[\textbf{U}] \right) \textbf{U}^{(\alpha)}, \mathcal{S}  \textbf{U}^{(\alpha)} \right)\leq \epsilon C \left(\frac{1}{\mu_{\min}}, \mu_{\max} \right) \sqrt{\mathcal{E}^{2} \left(\textbf{U} \right)} \mathcal{E}^{N} \! \left(\textbf{U} \right).
\end{equation}
\color{black}

\noindent Furthermore, for the bottom contribution, we easily get 

\small
\begin{equation}\label{estim_3}
\left\lvert \left( \frac{D}{\tanh(\sqrt{\mu} D)} \partial_{x}^{\alpha} \zeta, \partial_{x} \partial_{x}^{\alpha} \text{sech} \left(\sqrt{\mu} D \right) \left(b \; \text{sech} \left(\sqrt{\mu} D \right) v \right) \right) \right\rvert \leq C \left(\frac{1}{\mu_{\min}} \right) \lver b \rver_{\infty} \mathcal{E}^{N} \! \left(\textcolor{black}{\textbf{U}} \right). 
\end{equation}
\normalsize 

\noindent Finally, \textcolor{black}{from Equations \eqref{estim_1}, \eqref{estim_2}, \eqref{estim_3}}, we obtain, 

\begin{equation*}
\mathcal{E}^{N} \! \left(\textbf{U} \right) \leq \mathcal{E}^{N} \! \left(\textbf{U}_{0} \right) + \delta C \left(\frac{1}{\mu_{\min}}, \mu_{\max} \right) \int_{0}^{t} \left(\mathcal{E}^{N} \! \left(\textcolor{black}{\textbf{U}} \right)^{\frac{3}{2}} + \mathcal{E}^{N} \! \left(\textcolor{black}{\textbf{U}} \right) \right)(s) ds,
\end{equation*}

\noindent and there exists a time $T>0$, such that, for all $t \in \left[0, \frac{T}{\delta} \right]$,

\begin{equation*}
\mathcal{E}^{N} \! \left(\textbf{U}(t,\cdot) \right) \leq C \left(\frac{1}{\mu_{\min}}, \mu_{\max},  \mathcal{E}^{N} \!\! \left(\textbf{U}_{0} \right) \right).
\end{equation*}

\noindent The energy estimate follows from the Gronwall Lemma.

\end{proof}

\noindent In order to use a Lady Windermere's fan argument (\textcolor{black}{a well-known telescopic identity used to relate global and the local error)}, to prove the convergence of the numerical scheme, we need a  Lipschitz property for the flow of the Saut-Xu system \eqref{saut-xu-deep}.  We first give a control of the differential of the flow with respect to the initial datum.

\begin{prop}\label{lip_saut_diff}
Let $N \geq 2$, \upshape$\textbf{V}_{0} \in H^{N+\frac{1}{2}} \left(\R \right) \times H^{N} \left(\R \right)$ \itshape, \upshape$\textbf{U}_{0} \in H^{N+1+\frac{1}{2}} \left(\R \right) \times H^{N+1} \left(\R \right)$ \itshape, \upshape and  $b \in L^{\infty} \left(\R \right)$\itshape. We assume that $\epsilon, \beta, \mu$ satisfy Condition \eqref{parameters_constraints} and 

\upshape
\begin{equation*}
\lver \textbf{V}_{0} \rver_{H^{N+\frac{1}{2}} \times H^{N}} + \lver \textbf{U}_{0} \rver_{H^{N+1+\frac{1}{2}} \times H^{N+1}} + \lver b \rver_{L^{\infty}} \leq M.
\end{equation*}
\itshape

\noindent Then, there exists a time $T = T \left(M, \frac{1}{\mu_{\min}}, \mu_{\max}  \right)$ independent of the parameters $\epsilon$, $\mu$ and $\beta$ such that  \upshape$\left(\Phi^{t} \right)' \left( \textbf{U}_{0} \right) \cdot \left(\textbf{V}_{0} \right)$ \itshape exists on $\left[0, \frac{T}{\delta} \right]$. Furthermore, we have, for all $0 \leq t \leq \frac{T}{\delta}$,

\upshape
\begin{equation*}
\lver \left(\Phi^{t} \right)' \left( \textbf{U}_{0} \right) \cdot \left(\textbf{V}_{0} \right) \rver_{H^{N+\frac{1}{2}} \times H^{N}} \leq  C \left(\frac{1}{\mu_{\min}}, \mu_{\max}, M \right) \lver \textbf{V}_{0} \rver_{H^{N+\frac{1}{2}} \times H^{N}}.
\end{equation*}
\itshape
\end{prop}

\begin{proof}
\noindent We denote by $\textbf{U}(t) = \left( \zeta(t), v(t) \right)$ the solution of the Saut-Xu system \eqref{saut-xu-deep} with initial data $\textbf{U}_{0}$. We denote also $\left(\eta(t), w(t) \right) = \left(\Phi^{t} \right)' \left( \textbf{U}_{0} \right) \cdot \left(\textbf{V}_{0} \right)$. Then, $\left(\eta, w \right)$ satisfy the following system

\begin{equation}\label{saut_xu_linearise}
\partial_{t} \begin{pmatrix} \eta \\ w \end{pmatrix} + \mathcal{L} \begin{pmatrix} \eta \\ w \end{pmatrix} + \epsilon \sqrt{\mu} \mathcal{N}[(\zeta,v)] \partial_{x} \begin{pmatrix} \eta \\ w \end{pmatrix} + \epsilon \sqrt{\mu} \mathcal{N}[(\eta,w)] \partial_{x} \begin{pmatrix} \zeta \\ v \end{pmatrix} = \beta \sqrt{\mu} \left(\partial_{x} \left(B_{\mu} w \right), 0 \right)^{t},
\end{equation}

\noindent where

\begin{equation*}
\mathcal{L} = \begin{pmatrix} 0 & - \frac{1}{\sqrt{\mu} \nu} \Hmu \\ \partial_{x} & 0 \end{pmatrix} \quad \text{   and   } \quad \mathcal{N}[(\zeta,v)] \hspace{-0.05cm} = \hspace{-0.05cm} \begin{pmatrix} \frac{1}{2} \Hmu \left(v \Hmu \cdot \right) + \frac{1}{2} v  & \Hmu \left(\zeta \Hmu \cdot \right)  + \zeta \\ - \frac{1}{2} \partial_{x} \zeta \Hmu & \frac{3}{2} v - \frac{1}{2} v \Hmu^{2} \end{pmatrix} \hspace{-0.05cm}.
\end{equation*}

\begin{equation*}
\end{equation*}

\noindent For $0 \leq \alpha \leq N$, we denote $\textbf{V}^{(\alpha)} = \left(\partial_{x}^{\alpha} \eta, \partial_{x}^{\alpha} w \right)$. Then, applying $\partial_{x}^{\alpha}$ to System \eqref{saut_xu_linearise}, we get

\begin{equation*}
\partial_{t} \textbf{V}^{(\alpha)} \hspace{-0.05cm} + \hspace{-0.05cm} \mathcal{L} \textbf{V}^{(\alpha)} \hspace{-0.05cm} + \hspace{-0.05cm} \frac{\epsilon \sqrt{\mu}}{2}  \mathds{1}_{\left\{ \alpha \neq 0 \right\}} \left( \mathcal{B}[\textbf{U}] \textbf{V}^{(\alpha)} + \mathcal{B}[\textbf{V}] \partial_{x}^{\alpha} \textbf{U} \right) = \sqrt{\mu} \beta \begin{pmatrix}  \partial_{x} \partial_{x}^{\alpha} \left(B_{\mu} w \right) \\ 0 \end{pmatrix} + \epsilon \sqrt{\mu} \color{black} \mathcal{J}^{\alpha} \color{black},
\end{equation*} 

\noindent where 

\begin{equation*}
\begin{aligned}
& \mathcal{B}[\textbf{U}] = \begin{pmatrix} \Hmu \left( v  \Hmu \partial_{x} \; \cdot \;  \right) + v \partial_{x} & \Hmu \left(  \; \cdot \; \Hmu \partial_{x} \zeta \right) - \partial_{x} \zeta \Hmu^{2} \\ -  \partial_{x} \zeta \Hmu \partial_{x} -  \Hmu \partial_{x} \zeta \partial_{x} & 3 v \partial_{x} - v \Hmu^{2} \partial_{x} \end{pmatrix},\\
&\color{black} \mathcal{J}^{\alpha} = - \partial_{x}^{\alpha} \left( \mathcal{N}[(\zeta,v)] \partial_{x} \begin{pmatrix} \eta \\ w \end{pmatrix} + \mathcal{N}[(\eta,w)] \partial_{x} \begin{pmatrix} \zeta \\ v \end{pmatrix} \right) + \frac{1}{2} \left( \mathcal{B}[\textbf{U}] \textbf{V}^{(\alpha)} + \mathcal{B}[\textbf{V}] \partial_{x}^{\alpha} \textbf{U} \right). \color{black}
\end{aligned}
\end{equation*}

\noindent Then, we can show, as in Paragraph IV. B in \cite{saut_xu}, that

\begin{equation}\label{control_J}
\lver \color{black} \mathcal{J}^{\alpha} \color{black} \rver_{2} + \lver \lver D \rver^{\frac{1}{2}} \color{black} \mathcal{J}^{\alpha} \color{black} \rver_{2} \leq \epsilon \sqrt{\mu} C \left(\frac{1}{\mu_{\min}} \right) \mathcal{E}^{N} \! \left( \textcolor{black}{\textbf{V}} \right).
\end{equation}

\noindent We recall that we can symmetrize $\mathcal{L}$ thanks to 

\begin{equation*}
\mathcal{S} = \begin{pmatrix} \frac{D}{\tanh(\sqrt{\mu} D)} & 0 \\ 0 & 1 \end{pmatrix}.
\end{equation*}

\noindent We define the energy associated to this symmetrizer

\begin{equation*}
F^{\alpha} \left( \textbf{V} \right) = \lver \sqrt{\frac{D}{\tanh(\sqrt{\mu} D)}} \partial_{x}^{\alpha} \eta \rver^{2}_{2} + \lver \partial_{x}^{\alpha} w \rver^{2}_{2}, \quad \text{   and   } \quad F^{N} \left( \textbf{V} \right) = \underset{0 \leq \alpha \leq N}{\sum} F^{\alpha} \left( \textbf{V} \right).
\end{equation*}

\noindent We have, for $\alpha \neq 0$,

\small
\begin{equation*}
\begin{aligned}
\frac{d}{dt} F^{\alpha} \left( \textbf{V} \right) &= \epsilon \sqrt{\mu} \left(\color{black} \mathcal{J}^{\alpha} \color{black}, \mathcal{S} \textbf{V}^{(\alpha)} \right) - \frac{\epsilon \sqrt{\mu}}{2} \left( \left( \mathcal{B}[\textbf{U}] \textbf{V}^{(\alpha)},\mathcal{S} \textbf{V}^{(\alpha)} \right) + \left( \mathcal{B}[\textbf{V}] \partial_{x}^{\alpha} \textbf{U},\mathcal{S} \textbf{V}^{(\alpha)} \right) \right)\\
&\hspace{7.5cm} + \beta \sqrt{\mu} \left(\partial_{x} \partial_{x}^{\alpha} \left(B_{\mu} v \right),\mathcal{S} \textbf{V}^{(\alpha)}   \right)\\
&= I + II + III + IIII.
\end{aligned}
\end{equation*}
\normalsize

\noindent We can estimate \textit{I} thanks to estimate \eqref{control_J} and \textit{II} as in Paragraph IV. B in \cite{saut_xu}. For \textit{IIII}, we can proceed as in the previous theorem. For \textit{III}, we get, thanks to Proposition \ref{control_Hmu},

\begin{equation*}
\lver III \rver \leq \epsilon \sqrt{\mu} \lver \left( \zeta, v \right) \rver_{H^{N+1+\frac{1}{2}} \times H^{N+1}} \lver \left( \eta, w \right) \rver_{H^{N+\frac{1}{2}} \times H^{N}} 
\end{equation*}

\noindent Then, we obtain

\begin{equation*}
\frac{d}{dt} F^{N} \left( \textbf{V} \right) \leq \delta \sqrt{\mu} C(M) \left( F^{N} \left( \textbf{V} \right) + \sqrt{F^{N} \left(\textbf{V} \right)} \right),
\end{equation*}

\noindent and the result follows.

\end{proof}

\begin{prop}\label{lip_saut_xu}
Let $N \geq 2$, \upshape$\textbf{U}_{0}, \textbf{V}_{0} \in H^{N+1+\frac{1}{2}} \times H^{N+1} \left(\R \right)$ \itshape and \upshape $b \in L^{\infty} \left(\R \right)$\itshape. We assume that $\epsilon, \beta, \mu$ satisfy Condition \eqref{parameters_constraints} and 

\upshape
\begin{equation*}
\lver \textbf{V}_{0} \rver_{H^{N+1+\frac{1}{2}} \times H^{N+1}} + \lver \textbf{U}_{0} \rver_{H^{N+1+\frac{1}{2}} \times H^{N+1}} + \lver b \rver_{L^{\infty}} \leq M.
\end{equation*}
\itshape

\noindent Then, there exists a time $T$ independent of $\epsilon$, $\mu$ and $\beta$ and two unique solutions \upshape$\textbf{U}, \textbf{V}$ \itshape of the system \eqref{saut-xu-deep} on $\left[0, \frac{T}{\delta} \right]$ with initial data \upshape$\textbf{U}_{0}$ \itshape and \upshape$\textbf{V}_{0}$\itshape. Furthermore, we have the following Lipschitz estimate, for all $0 \leq t \leq \frac{T}{\delta}$,

\upshape
\begin{equation}\label{ConstLip}
\lver \textbf{U}(t, \cdot) -  \textbf{V}(t, \cdot) \rver_{H^{N+\frac{1}{2}} \times H^{N}} \leq  K \lver \textbf{U}_{0} -  \textbf{V}_{0} \rver_{H^{N+\frac{1}{2}} \times H^{N}},
\end{equation}
\itshape

\noindent where $K = C \left(\frac{1}{\mu_{\min}}, \mu_{\max}, M \right)$.
\end{prop}

\begin{proof}
\noindent The existence of $\textbf{U}, \textbf{V}$ and $T$ follow from the previous theorem. Furthermore, we have

\begin{equation*}
\textbf{U}(t) - \textbf{V}(t) = \int_{s=0}^{1} \left(\Phi^{t} \right)' \left(\textbf{V}_{0} + s \left(\textbf{U}_{0} - \textbf{V}_{0} \right) \right) \cdot \left(\textbf{U}_{0} - \textbf{V}_{0} \right).
\end{equation*}

\noindent The result follows from Proposition \ref{lip_saut_diff}.
\end{proof}

\section{A splitting scheme}\label{proof_splitting}

In this section, we  split the  Saut-Xu system  \eqref{saut-xu-deep} and we give some estimates for the sub-problems. We consider, separately, the transport part  

\begin{equation}\label{transport_part}
\left\{
\begin{array}{l}
\partial_{t} \zeta +\frac{\epsilon \sqrt{\mu}}{2} \left( \left(\Hmu^{2} + 1 \right) v \right) \partial_{x} \zeta = 0\\
\partial_{t} v + \frac{3 \epsilon \sqrt{\mu}}{2} v \partial_{x} v = 0,
\end{array}
\right.
\end{equation}

\noindent and the dispersive part

\begin{equation}\label{dispersive_part}
\left\{
\begin{array}{l}
\hspace{-0.1cm} \partial_{t} \zeta \hspace{-0.05cm} - \hspace{-0.05cm} \Hmu v \hspace{-0.05cm} + \hspace{-0.05cm} \epsilon \sqrt{\mu} \left( \frac{1}{2} \Hmu \left(v \partial_{x} \Hmu \zeta \right) \hspace{-0.05cm} + \hspace{-0.05cm} \Hmu \left( \zeta \partial_{x} \Hmu v \right) \hspace{-0.05cm} + \hspace{-0.05cm} \zeta \partial_{x} v \hspace{-0.05cm} - \hspace{-0.05cm} \frac{1}{2} \partial_{x} \zeta \Hmu^{2} v \right) = \beta \sqrt{\mu} \partial_{x} \left(B_{\mu} v \right)\\
\hspace{-0.1cm} \partial_{t} v \hspace{-0.05cm} + \hspace{-0.05cm} \partial_{x} \zeta \hspace{-0.05cm} - \hspace{-0.05cm} \frac{\epsilon \sqrt{\mu}}{2} \partial_{x} \zeta \Hmu \partial_{x} \zeta - \frac{\epsilon \sqrt{\mu}}{2} v \Hmu^{2} \partial_{x} v = 0. \\
\end{array}
\right.
\end{equation}

\noindent We denote by $\Phi_{\mathcal{A}}^{t}$ the flow of System \eqref{transport_part} and by $\Phi_{\mathcal{D}}^{t}$ the flow of System \eqref{dispersive_part}.

\begin{remark}
\label{Tdecomposition}
\noindent Notice that we keep the term $\zeta \partial_{x} v$ in the first equation and we decompose $v \partial_{x} \zeta$ as $v \partial_{x} \zeta = \partial_{x} \zeta \left(\Hmu^{2} +1 \right)v - \partial_{x} \zeta \Hmu^{2} v$ . This will be useful for the local wellposedness of the dispersive part. 
\end{remark}

\noindent In the following, we prove the local existence on large time for Systems \eqref{transport_part} and \eqref{dispersive_part}.

\subsection{The transport equation}

\noindent The system \eqref{transport_part} is a transport equation. Then, it is easy to get the following result.

\begin{prop}\label{existence_transport}
Let $s_{1} \geq 0$, $s_{2} > \frac{3}{2}$ and $M > 0$. We assume that $\epsilon, \mu$ satisfies Condition \eqref{parameters_constraints}. Then, there exists a time $T_{1} = T_{1} \left(M,\mu_{\max} \right) >0$, such that if 

\upshape
\begin{equation*}
\lver \zeta_{0} \rver_{H^{s_{1}}} + \lver v_{0} \rver_{H^{s_{2}}} \leq M,
\end{equation*}
\itshape

\noindent we have a unique solution $\left(\zeta, v \right) \in \mathcal{C} \left(\left[0, \frac{T_{1}}{\epsilon} \right], H^{s_{1}}(\R) \times H^{s_{2}}(\R) \right)$, to System \eqref{transport_part} with initial data $\left(\zeta_{0}, v_{0} \right)$. Furthermore, we have, for all $t\leq \frac{T_{1}}{\epsilon}$,

\begin{equation}\label{estim_int1}
\lver \zeta(t,\cdot) \rver_{H^{s_{1}}} + \lver v(t,\cdot) \rver_{H^{s_{2}}} \leq C(M,\mu_{\max}).
\end{equation}

\noindent Finally, if $s_{2} \geq 4$ and $M_{1} = \underset{0 \leq t \leq \frac{T_{1}}{\epsilon}}{\max} \lver v(t,\cdot) \rver_{H^{s_{2}-2}}$, then for all $t\leq \frac{T_{1}}{\epsilon}$, we have

\upshape
\begin{equation}\label{borne1}
\lver \zeta(t,\cdot) \rver_{H^{s_{1}}} + \lver v(t,\cdot) \rver_{H^{s_{2}}} \leq e^{\epsilon C_{1} t} |\Ub_0|_{H^{s_1}\times H^{s_2}}, 
\end{equation}
\itshape

\noindent where $C_{1}>0$ depends on $M_{1}$ and $\mu_{\max}$.
\end{prop}

\begin{proof}

\noindent The proof follows from the fact that the quasilinear system \eqref{transport_part} is symmetric. Thanks to the Coifman-Meyer estimate (see Proposition \ref{coifmanmeyer}), we get

\begin{equation*}
\frac{d}{dt} \left(\lver \zeta(t,\cdot) \rver_{H^{s_{1}}}^{2} + \lver v(t,\cdot) \rver_{H^{s_{2}}}^{2} \right) \leq C \epsilon \sqrt{\mu} \left( \lver \zeta(t,\cdot) \rver_{H^{s_{1}}}^{2} + \lver v(t,\cdot) \rver_{H^{s_{2}}}^{2} \right)^{\frac{3}{2}}.
\end{equation*}

\noindent Then, we see that the energy is bounded uniformly with respect to $\epsilon$ and $\mu$ and we get Estimate \eqref{estim_int1}. For the second estimate, using the same trick that in Lemma 3.1 in \cite{HoldenLR13}, we notice that, if $s_{1} \geq 4$, 

\begin{equation*}
\frac{d}{dt} \left(\lver \zeta(t,\cdot) \rver_{H^{s_{1}}}^{2} + \lver v(t,\cdot) \rver_{H^{s_{2}}}^{2} \right) \leq \epsilon \sqrt{\mu}  \lver v(t,\cdot) \rver_{H^{s_{2}-2}} \left( \lver \zeta(t,\cdot) \rver_{H^{s_{1}}}^{2} + \lver v(t,\cdot) \rver_{H^{s_{2}}}^{2} \right).
\end{equation*}

\noindent By applying the Gronwall lemma, we get the result.
\end{proof}

\subsection{The dispersive equation}

\medskip

\noindent The system \eqref{dispersive_part} contains all the dispersive terms of the Saut-Xu system. We have the following estimate for the flow. 

\begin{prop}\label{existence_dispersive}
Let $N \geq 2$, and \upshape $b \in L^{\infty}(\R)$\itshape. We assume that $\epsilon, \beta, \mu$ satisfy Condition \eqref{parameters_constraints}. Then,  there exists a time $T_{2} = T_{2} \left(M, \frac{1}{\mu_{\min}}, \mu_{\max}  \right)$ such that if

\begin{equation*}
\lver \zeta_{0} \rver_{H^{N+\frac{1}{2}}} + \lver v_{0} \rver_{H^{N}} + \lver b \rver_{L^{\infty}} \leq M,
\end{equation*}

\noindent we have a unique solution $\left(\zeta,v \right) \in \mathcal{C} \left(\left[0, \frac{T_{2}}{\delta} \right], H^{N+\frac{1}{2}} (\R) \times H^{N}(\R) \right)$ \itshape to the system \eqref{dispersive_part} with initial data \upshape$\left(\zeta_{0},v_{0} \right)$\itshape. Furthermore, we have, for all $t \leq \frac{T_{2}}{\delta}$,

\upshape
\begin{equation}\label{estim_int2}
\lver \zeta(t, \cdot) \rver_{H^{N+\frac{1}{2}}} + \lver v(t,\cdot) \rver_{H^{N}} \leq C \left( M,\mu_{\max}, \frac{1}{\mu_{\min}} \right).
\end{equation}
\itshape 

\noindent Finally, if \textcolor{black}{$N \geq 7$}, and

\upshape
\begin{equation*}
M_{1} = \underset{0 \leq t \leq \frac{T_{2}}{\delta}}{\max} \left( \lver \zeta(t, \cdot) \rver_{H^{N+\frac{1}{2}-2}} + \lver v(t,\cdot) \rver_{H^{N-2}} \right),
\end{equation*}
\itshape

\noindent then for all $t\leq \frac{T_{2}}{\delta}$, we have

\upshape
\begin{equation}\label{borne2}
\lver \zeta(t, \cdot) \rver_{H^{N+\frac{1}{2}}} + \lver v(t,\cdot) \rver_{H^{N}} \leq e^{\delta C_{2} t} |\Ub_0|_{H^{N+1/2} \times H^N },
\end{equation}
\itshape

\noindent where $C_{2}$ is a positive constant which depends on $\mu_{\max}, \frac{1}{\mu_{\min}}, M_{1}$.  
\end{prop}

\begin{proof}
\noindent The proof is an adaptation of the proof of Theorem \ref{existence_saut_xu} and part IV in \cite{saut_xu}. We notice that, in the proof of Saut and Xu, the transport part can be treated separately and does not influence the control of the other terms. \textcolor{black}{Hence, we can use the same symmetrizer $\mathcal{S}$ that in Theorem \ref{existence_saut_xu} (see \eqref{symmetrizer})  and we get}

\begin{equation*}
\frac{d}{dt} \mathcal{E}^{N} \! \left(\zeta,v \right) \leq C \left(\frac{1}{\mu_{\min}} \right)  \left(\frac{\epsilon}{\nu} \mathcal{E}^{N} \! \left(\zeta,v \right)^{\frac{3}{2}} + \frac{\beta}{\nu} \mathcal{E}^{N} \! \left(\zeta,v \right) \right).
\end{equation*}

\noindent Then, by Remark \ref{explication_assumption_mu}, we get Estimate \eqref{estim_int2}. \color{black} Furthermore, we notice that we use the same trick as in Lemma 3.1 in \cite{HoldenLR13}. By keeping the same notations as in Theorem \ref{existence_saut_xu}, we get from Equations \eqref{estim_2} and \eqref{estim_3} that

\color{black}

\begin{equation*}
\frac{d}{dt} \mathcal{E}^{N} \! \left(\textbf{U} \right) \leq \delta C \left(\frac{1}{\mu_{\min}}, \mu_{\max} \right)  \left( \left\lvert \tilde{\mathcal{G}}^{N} \right\rvert_{ H^{1/2}\times L^{2}} + \sqrt{\mathcal{E}^{2} \left(\textbf{U} \right)} + \lver b \rver_{\infty} \right) \mathcal{E}^{N} \! \left(\textbf{U} \right)
\end{equation*}

\noindent where $\tilde{\mathcal{G}}^{N} = (\tilde{\mathcal{G}}^{N}_{1}, \tilde{\mathcal{G}}^{N}_{2})^{t}$ with

\small
\begin{equation*}
\begin{aligned}
&\tilde{\mathcal{G}}^{N}_{1} = - \partial_{x}^{N} \left( [\Hmu, \zeta] \Hmu \partial_{x} v + \zeta (\Hmu^{2}+1) \partial_{x} v \right)  -  \frac{1}{2} \underset{1 \leq \gamma \leq N-1}{\sum} \hspace{-0.4cm} C^{\gamma}_{N} \Hmu ( \partial_{x}^{\gamma} v \Hmu \partial_{x}^{1+N-\gamma} \zeta) -  \frac{1}{2} \partial_{x} \zeta (\Hmu^{2}+1) \partial_{x}^{N} v\\
&\tilde{\mathcal{G}}^{N}_{2} = \frac{1}{2} \underset{1 \leq \gamma \leq N-1}{\sum} \hspace{-0.3cm} C^{\gamma}_{N} \partial_{x}^{1+\gamma} \zeta \Hmu \partial_{x}^{1+N-\gamma} \zeta + \frac{1}{2}  \underset{1 \leq \gamma \leq N}{\sum} C^{\gamma}_{N} \partial_{x}^{\gamma} v \Hmu^{2} \partial_{x}^{1+N-\gamma} v.
\end{aligned}
\end{equation*}
\normalsize

\noindent To explain how we can adapt the trick used in Lemma 3.1 in \cite{HoldenLR13}, we focus our attention to one term. For $1 \leq \gamma \leq N-1$, we have to control $\left\lvert \partial_{x}^{1+\gamma} \zeta \Hmu \partial_{x}^{1+N-\gamma} \zeta \right\rvert_{L^2}$. If $\gamma \leq \lfloor \frac{N}{2} \rfloor$, we get from Propositions \ref{productestim} and \ref{control_Hmu} that

\begin{equation*}
\left\lvert \partial_{x}^{1+\gamma} \zeta \Hmu \partial_{x}^{1+N-\gamma} \zeta \right\rvert_{L^2} \leq \left\lvert \partial_{x}^{1+\gamma} \zeta \right\rvert_{H^{1}} \left\lvert \Hmu \partial_{x}^{1+N-\gamma} \zeta \right\rvert_{L^2} \leq C \left(\mu_{\max} \right) \left\lvert \zeta \right\rvert_{H^{2 + \lfloor \frac{N}{2} \rfloor}} \left\lvert \zeta \right\rvert_{H^{N}},
\end{equation*}

\noindent whereas if $\gamma > \lfloor \frac{N}{2} \rfloor$, we have

\begin{equation*}
\left\lvert \partial_{x}^{1+\gamma} \zeta \Hmu \partial_{x}^{1+N-\gamma} \zeta \right\rvert_{L^2} \leq \left\lvert \partial_{x}^{1+\gamma} \zeta \right\rvert_{L^2} \left\lvert \Hmu \partial_{x}^{1+N-\gamma} \zeta \right\rvert_{H^1} \leq C \left(\mu_{\max} \right) \left\lvert \zeta \right\rvert_{H^{N}} \left\lvert \zeta \right\rvert_{H^{2+ \lfloor \frac{N}{2} \rfloor}}.
\end{equation*}

\noindent We can mimic this method to control the other terms of $\tilde{\mathcal{G}}^{N}$ and, thanks to Propositions \ref{control_Hmu}, \ref{Hmu_commut} and \ref{productestim}, we obtain if $N \geq 7$ that

\begin{equation*}
\frac{d}{dt} \mathcal{E}^{N} \! \left(\zeta,v \right) \leq \delta C \left(\frac{1}{\mu_{\min}}, \mu_{\max} \right)  \left( \lver \zeta(t, \cdot) \rver_{H^{N+\frac{1}{2}-2}} + \lver v(t,\cdot) \rver_{H^{N-2}} + \left\lvert b \right\rvert_{L^{\infty}}  \right) \mathcal{E}^{N} \! \left(\zeta,v \right).
\end{equation*}
\color{black}

\noindent Then, Estimate \eqref{borne2} follows.
\end{proof}

\begin{remark}\label{rem_op_split}
\noindent Under the assumption of Proposition \ref{existence_dispersive} and if \textcolor{black}{$N \geq 7$}, we get from relations \eqref{borne1} and \eqref{borne2} that, there exists a time $T_{3}>0$, such that for all $t\in \left[0, \frac{T_{3}}{\delta} \right]$,
 
\upshape
\begin{equation*}
|\mathcal{Y}^t \Ub_0 |_{ H^{N+1/2}\times H^{N}} \leq e^{C_{3} \delta t} |\Ub_0|_{H^{N+1/2}\times H^{N}},
\end{equation*}
\itshape

\noindent where $C_{3} = C \left(\lver \Ub_0 \rver_{H^{N+1/2-2}\times H^{N-2}}, \mu_{\max}, \frac{1}{\mu_{\min}} \right)$ and $T_{3} = C \left( \lver \Ub_0 \rver_{H^{N+1/2}\times H^{N}}, \mu_{\max}, \frac{1}{\mu_{\min}} \right)$.
\end{remark}

\section{Error estimates}\label{mainsection}
\noindent The goal of this part is to prove the main result of this paper (Theorem \ref{thm_global_error}). Our analysis is based on energy estimates.

\subsection{The local error estimate}

\noindent The local error is the following quantity

\begin{equation}\label{localerror}
e \left(t, \textbf{U}_{0}  \right) = \Phi^{t}   \textbf{U}_{0} -  \mathcal{Y}^{t}  \textbf{U}_{0}.
\end{equation}

\noindent Our approach is similar to the one developed in \cite{split_chartier}. We use the fact that $\Phi^{t} \textbf{U}_{0}$ satisfies a symmetrizable system. Therefore, $e$ satisfies this system up to a remainder and then, we can control $e$ thanks to energy estimates. In the following we give different technical lemmas in order to control the local error. We recall that the transport operator is the operator $\mathcal{A}$

\begin{equation*}
\mathcal{A} \left(\zeta, v \right) = - \frac{\epsilon \sqrt{\mu}}{2} \begin{pmatrix} \left(\left(\Hmu^{2} + 1 \right) \! v \right) \partial_{x} \zeta \\ 3 v \partial_{x} v \end{pmatrix}.
\end{equation*}

\noindent The following proposition gives an estimate of the differential of the transport operator.

\begin{lemma}\label{Aprime}
Let $s_{1}, s_{2} \geq 0$ and $\epsilon, \mu$ satisfying Condition \eqref{parameters_constraints}. Then,

\begin{equation*}
\lver \mathcal{A}' (\zeta,v).(\eta,w) \rver_{H^{s_{1}} \times H^{s_{2}}} \leq \epsilon C(\mu_{\max}) \lver (\zeta,v) \rver_{H^{s_{1}+1} \times H^{s_{2}+1}} \lver (\eta,w) \rver_{H^{s_{1}+1} \times H^{s_{2}+1}}.
\end{equation*}

\end{lemma}

\begin{proof}
\noindent We have 

\begin{equation*}
\mathcal{A}' (\zeta,v).(\eta,w) =  - \frac{\epsilon \sqrt{\mu}}{2} \begin{pmatrix}  \left( \left(\Hmu^{2} + 1 \right) \! v \right) \partial_{x} \eta + \left( \left(\Hmu^{2} + 1 \right) \! w  \right)\partial_{x} \zeta  \\ 3 w \partial_{x} v+ 3 v \partial_{x} w \end{pmatrix},
\end{equation*}

\noindent and the estimate follows from \textcolor{black}{Propositions \ref{control_Hmu} and \ref{productestim}.}
\end{proof}

\noindent We can do the same for the dispersive part (using also Proposition \ref{control_Hmu}). We recall that the dispersive operator is the operator $\mathcal{D}$

\small
\begin{equation*}
\mathcal{D}(\zeta,v) = \begin{pmatrix} \Hmu v  + \epsilon \sqrt{\mu} \left( \frac{1}{2} \Hmu \left(v \partial_{x} \Hmu \zeta \right) + \Hmu \left( \zeta \partial_{x} \Hmu v \right) +  \zeta \partial_{x} v - \frac{1}{2} \partial_{x} \zeta \Hmu^{2} v \right) - \beta \sqrt{\mu} \partial_{x} \left(B_{\mu} v \right) \\ - \partial_{x} \zeta + \frac{\epsilon \sqrt{\mu}}{2} \partial_{x} \zeta \Hmu \partial_{x} \zeta + \frac{\epsilon \sqrt{\mu}}{2} v \Hmu^{2} \partial_{x} v \end{pmatrix}
\end{equation*}
\normalsize

\begin{lemma}\label{Dprime}
Let $s >0$,  $\epsilon,\beta, \mu$ satisfying Condition \eqref{parameters_constraints}  and $b \in L^{\infty}(\R)$. Then,

\begin{equation*}
\lver \mathcal{D}' (\zeta,v).(\eta,w) \rver_{H^{s} \times H^{s}} \hspace{-0.05cm} \leq \hspace{-0.05cm} C(\mu_{\max}) \left(1 \hspace{-0.1cm} + \hspace{-0.1cm} \beta \lver b \rver_{L^{\infty}} \hspace{-0.1cm} + \hspace{-0.1cm} \epsilon \lver (\zeta,v) \rver_{H^{s+1} \times H^{s+1}} \right) \hspace{-0.1cm} \lver (\eta,w) \rver_{H^{s+1} \times H^{s+1}} \hspace{-0.05cm}.
\end{equation*}
\end{lemma}

\noindent Furthermore, we have to control the derivative of the flow $\Phi_{\mathcal{A}}^{t}$ with respect to the initial data. We denote it by $\left(\Phi_{\mathcal{A}}^{t} \right)'$. 

\begin{lemma}\label{dphiAprime}
Let $s_{1}, s_{2} \geq 0$, $M > 0$,  $\epsilon,\beta, \mu$ satisfying Condition \eqref{parameters_constraints}  and $b \in L^{\infty}(\R)$. Let $\left( \zeta_{0}, v_{0} \right) \in H^{s_{1}+1} \times H^{s_{2}+1} (\RD)$ such that,

\begin{equation*}
\lver \left( \zeta_{0}, v_{0} \right) \rver _{H^{s_{1}+1} \times H^{s_{2}+1}} \leq M.
\end{equation*}

\noindent Then, there exists a time $T = T(M,\mu_{\max})$, such that $\left(\Phi_{\mathcal{A}}^{t} \right)' \left( \zeta_{0}, v_{0} \right) \cdot \left(\eta_{0},w_{0} \right)$ exists for all $t \in \left[0, \frac{T}{\delta} \right]$ and if we denote

\begin{equation*}
\begin{pmatrix} \eta \\ w \end{pmatrix} = \left(\Phi_{\mathcal{A}}^{t} \right)' \left( \zeta_{0}, v_{0} \right) \cdot \left(\eta_{0},w_{0} \right),
\end{equation*}

\noindent for all  $0 \leq t \leq \frac{T}{\delta }$,

\begin{equation*}
\lver \left( \eta, w \right)(t,\cdot) \rver_{H^{s_{1}} \times H^{s_{2}}} \leq \lver \left( \eta_{0}, w_{0} \right) \rver _{H^{s_{1}} \times H^{s_{2}}} C \left(\mu_{\max}, M \right).
\end{equation*}

\end{lemma}

\begin{proof}
\noindent The quantity $\left( \eta, w \right)$ satisfies the following linear system

\begin{equation*}
\left\{
\begin{array}{l}
\partial_{t} \eta + \frac{\epsilon \sqrt{\mu}}{2} \left(\Hmu^{2} + 1 \right) v \partial_{x} \eta +  \frac{\epsilon \sqrt{\mu}}{2} \left(\Hmu^{2} + 1 \right) w \partial_{x} \zeta = 0,\\
\partial_{t} w + \frac{3 \epsilon \sqrt{\mu}}{2} v \partial_{x} w +  \frac{3 \epsilon \sqrt{\mu}}{2} w \partial_{x} v = 0,
\end{array}
\right.
\end{equation*}

\noindent where $\left( \zeta, v \right) = \Phi_{\mathcal{A}}^{t} \left( \zeta_{0}, v_{0} \right)$. The result follows from energy estimates, the Gronwall lemma and Proposition \ref{existence_transport}.
\end{proof}

\noindent In the following, we use the fact $\Phi^{t}_{\mathcal{A}} \circ \Phi^{t}_{\mathcal{D}}$ satisfies the Saut-Xu system \eqref{saut-xu-deep} up to a remainder. The following lemma is the key point for the control of this remainder.

\begin{lemma}\label{controlremainder}
\noindent Let $N \geq 2$, $M > 0$,  $\epsilon,\beta, \mu$ satisfying Condition \eqref{parameters_constraints}  and $b \in L^{\infty}(\R)$. Let \upshape$\textbf{U} = \left( \zeta, v \right) \in H^{N+\frac{1}{2}} \times H^{N} (\RD)$ \itshape such that,

\upshape
\begin{equation*}
\lver b \rver_{L^{\infty}} + \lver\textbf{U} \rver _{H^{N+\frac{1}{2}} \times H^{N} (\R)} \leq M.
\end{equation*}
\itshape

\noindent Then, there exists a time $T = T \left(M, \mu_{\max}, \frac{1}{\mu_{\min}} \right) > 0$, such that \upshape$\Phi^{t}_{\mathcal{A}} \left( \textbf{U} \right)$ \itshape exists for all  $0 \leq t \leq \frac{T}{\delta}$, and furthermore,

\upshape
\begin{equation*}
\lver \left(\Phi^{t}_{\mathcal{A}} \right)' \left(\textbf{U} \right) \cdot \mathcal{D} \left(\textbf{U} \right) - \mathcal{D} \left( \Phi_{\mathcal{A}}^{t} \left(\textbf{U} \right) \right) \rver_{H^{N-2} \times H^{N-2}} \leq \epsilon C \left(M, \mu_{\max}, \frac{1}{\mu_{\min}} \right) t. 
\end{equation*}
\itshape

\end{lemma}

\begin{proof}
\noindent The existence of $T$ follows from Proposition \ref{existence_transport}. Then, we notice that

\begin{equation*}
\left(\Phi^{t}_{\mathcal{A}} \right)' \hspace{-0.1cm} \left(\textbf{U} \right) \hspace{-0.05cm} \cdot \mathcal{D} \hspace{-0.05cm} \left(\textbf{U} \right) - \mathcal{D} \hspace{-0.05cm} \left( \hspace{-0.05cm} \Phi_{\mathcal{A}}^{t} \hspace{-0.05cm} \left(\textbf{U} \right) \right) \hspace{-0.1cm} = \hspace{-0.15cm} \int_{0}^{t} \hspace{-0.25cm} \mathcal{A}' \hspace{-0.1cm} \left(\Phi_{\mathcal{A}}^{s} \hspace{-0.05cm} \left(\textbf{U} \right) \right) \cdot \left( \left(\Phi^{s}_{\mathcal{A}} \right)' \hspace{-0.1cm} \left(\textbf{U} \right) \hspace{-0.05cm} \cdot \hspace{-0.1cm} \mathcal{D} \hspace{-0.05cm} \left( \textbf{U} \right) \hspace{-0.05cm} \right) - \mathcal{D}' \hspace{-0.05cm} \left( \hspace{-0.05cm} \Phi_{\mathcal{A}}^{s} \left(\textbf{U} \right) \hspace{-0.05cm} \right) \cdot \mathcal{A} \left( \Phi^{s}_{\mathcal{A}} \left( \textbf{U} \right) \hspace{-0.05cm} \right) \hspace{-0.05cm}.
\end{equation*}

\noindent Using Lemmas \ref{Aprime}, \ref{Dprime} and Proposition \ref{existence_transport}, we get,

\small
\begin{align*}
 \lver \left(\Phi^{t}_{\mathcal{A}} \right)' \hspace{-0.1cm} \left(\textbf{U} \right) \hspace{-0.05cm} \cdot \mathcal{D} \hspace{-0.05cm} \left(\textbf{U} \right) - \mathcal{D} \hspace{-0.05cm} \left( \hspace{-0.05cm} \Phi_{\mathcal{A}}^{t} \hspace{-0.05cm} \left(\textbf{U} \right) \right) \rver_{H^{N-2} \times H^{N-2}} \leq  C \left( \mu_{\max}, M \right) \int_{0}^{t} &\epsilon \lver \left(\Phi^{s}_{\mathcal{A}} \right)' \left(\textbf{U} \right) \hspace{-0.05cm} \cdot \mathcal{D} \left( \textbf{U} \right) \rver_{H^{N-1} \times H^{N-1}}\\
& + \lver \mathcal{A} \left( \Phi^{s}_{\mathcal{A}} \left( \textbf{U} \right) \hspace{-0.05cm} \right) \rver_{H^{N-1} \times H^{N-1}}.
\end{align*}
\normalsize

\noindent Then, using Lemma \ref{dphiAprime}, the product estimate \ref{productestim} and the expression of $\mathcal{A}$, we obtain

\small
\begin{equation*}
 \lver \hspace{-0.05cm} \left(\Phi^{t}_{\mathcal{A}} \right)' \hspace{-0.15cm} \left(\textbf{U} \right) \hspace{-0.1cm} \cdot \hspace{-0.05cm} \mathcal{D} \hspace{-0.05cm} \left(\textbf{U} \right) \hspace{-0.05cm} - \hspace{-0.05cm} \mathcal{D} \hspace{-0.05cm} \left( \hspace{-0.05cm} \Phi_{\mathcal{A}}^{t} \hspace{-0.05cm} \left(\textbf{U} \right) \hspace{-0.05cm} \right) \hspace{-0.05cm} \rver_{H^{N-2} \times H^{N-2}} \hspace{-0.15cm} \leq \hspace{-0.05cm} \epsilon C \hspace{-0.05cm} \left( \mu_{\max}, M \right) \hspace{-0.15cm} \int_{0}^{t} \hspace{-0.2cm} \lver \mathcal{D} \left( \textbf{U} \right) \rver_{H^{N-1} \times H^{N-1}} + \lver \Phi^{s}_{\mathcal{A}} \left( \textbf{U} \right) \rver^{2}_{H^{N} \times H^{N}} \hspace{-0.1cm}.
\end{equation*}
\normalsize

\noindent Finally, the result follows from the expression of $\mathcal{D}$, the product estimate \ref{productestim} and Proposition \ref{control_Hmu}.
\end{proof}

\noindent We can now give the main result of this part, the local error estimate. 

\begin{prop}\label{control_localerror}
Let $N \geq 4$, $M > 0$,  $\epsilon,\beta, \mu$ satisfying Condition \eqref{parameters_constraints}  and $b \in L^{\infty}(\R)$. Let \upshape$\textbf{U}_{0} = \left( \zeta_{0}, v_{0} \right)$ \itshape such that,

\upshape
\begin{equation*}
\lver b \rver_{L^{\infty}} + \lver \textbf{U}_{0} \rver_{H^{N+\frac{1}{2}} \times H^{N}} \leq M.
\end{equation*}
\itshape

\noindent Then, there exists a time $T_{4} = T_{4} \left(M, \frac{1}{\mu_{\min}}, \mu_{\max} \right) > 0$, such that the local error \upshape$e \left(t, \textbf{U} \right)$ defined in \itshape \eqref{localerror} exists for all  $0 \leq t \leq \frac{T_{4}}{\delta }$, and furthermore,

\upshape
\begin{equation*}
\lver e \left(t, \textbf{U}_{0} \right) \rver_{H^{N-4 + \frac{1}{2}} \times H^{N-4}} \leq \delta C_{4} t^{2}, 
\end{equation*}
\itshape

\noindent where $C_{4} = C \left(\frac{1}{\mu_{\min}}, \mu_{\max}, M \right)$.
\end{prop}

\begin{proof}
\noindent From Propositions \ref{existence_transport} and \ref{existence_dispersive}, we obtain the existence of $T$. We denote 

\begin{equation*}
\textbf{U}(t) = \begin{pmatrix} \zeta(t) \\ v(t) \end{pmatrix} = \Phi^{t} \left(\textbf{U}_{0} \right) \text{   and   }\textbf{V}(t) = \begin{pmatrix} \eta(t) \\ w(t) \end{pmatrix} = \Phi^{t}_{\mathcal{A}} \left( \Phi^{t}_{\mathcal{D}} \left(\textbf{U}_{0} \right) \right).
\end{equation*}

\noindent Then, from Theorem \ref{existence_saut_xu} and Propositions \ref{existence_transport} and \ref{existence_dispersive}  we also have, for all $0 \leq t \leq \frac{T}{\delta}$,

\begin{equation}\label{controlUV}
\lver \textbf{U}(t, \cdot) \rver_{H^{N+\frac{1}{2}} \times H^{N}} + \lver  \textbf{V}(t, \cdot) \rver_{H^{N+\frac{1}{2}} \times H^{N}} \leq C \left(\frac{1}{\mu_{\min}}, \mu_{\max}, M \right).
\end{equation}

\noindent We know that $\left(\zeta, v \right)$ satisfy the Saut-Xu system \eqref{saut-xu-deep}. Furthermore, $\left(\eta, w \right)$ also satisfy the Saut-Xu system \eqref{saut-xu-deep} up to a remainder

\begin{equation*}
\partial_{t} \begin{pmatrix} \eta \\ w \end{pmatrix} = \mathcal{A} \left(\eta,w\right) +  \mathcal{D} \left(\eta,w\right) + \mathcal{R}(t),
\end{equation*}

\noindent where $\mathcal{R}(t) = \left(\Phi^{t}_{\mathcal{A}} \right)' \left(\Phi^{t}_{\mathcal{D}} \left(\textbf{U}_{0} \right) \right) \cdot \mathcal{D} \left(\Phi^{t}_{\mathcal{D}} \left(\textbf{U}_{0} \right) \right) - \mathcal{D} \left( \Phi_{\mathcal{A}}^{t} \left(\Phi^{t}_{\mathcal{D}} \left(\textbf{U}_{0} \right) \right) \right)$. Therefore, the local error $e$ satisfies the following system

\begin{equation}\label{eq_on_e}
\partial_{t} e = \begin{pmatrix} 0 & H_{\mu} \\- \partial_{x} & 0 \end{pmatrix} e +  \begin{pmatrix} 0 & \beta \sqrt{\mu} B_{\mu} \\ 0 & 0 \end{pmatrix} e + \mathcal{T}_{\mu} \left(\left(\zeta,v \right), \left(\eta,w \right)\right)  - \mathcal{R}(t),
\end{equation}

\noindent where  the operator $\mathcal{T}_{\mu} \left(\textbf{U}, \textbf{V} \right)$ is quadratic and satisfies the following estimate, for $0 \leq s \leq N-1$,

\begin{equation}\label{controlnonlin}
\lver \mathcal{T}_{\mu} \left(\left(\zeta,v \right), \left(\eta,w \right)\right) \rver_{H^{s} \times H^{s}} \leq \epsilon C \left(\frac{1}{\mu_{\min}}, \mu_{\max}, M \right) \lver e \rver_{H^{s+1} \times H^{s+1}}.
\end{equation}

\noindent Then, since $e_{|t=0}=0$,

\begin{equation*}
e(t,\cdot) = \int_{0}^{t} \partial_{t} e(s,\cdot) ds,
\end{equation*}

\noindent and since $e$ satisfies \eqref{eq_on_e}, we obtain, thanks to Estimates \eqref{controlUV}, \eqref{controlnonlin} and Lemma \ref{controlremainder},

\begin{equation}\label{firstenergyest}
\lver e \left(t, \cdot \right) \rver_{H^{N-2} \times H^{N-2}} \leq C \left(\frac{1}{\mu_{\min}}, \mu_{\max}, M \right) t.
\end{equation}

\noindent Furthermore, we recall that the Saut-Xu system \eqref{saut-xu-deep} is symmetrizable thanks to the symmetrizer (see Theorem \ref{existence_saut_xu})

\begin{equation*}
\mathcal{S} = \begin{pmatrix} \frac{D}{\tanh(\sqrt{\mu} D)} & 0 \\ 0 & 1 \end{pmatrix}.
\end{equation*} 

\noindent Therefore, applying $\mathcal{S}$ to the system \eqref{eq_on_e}, and using the fact that $\sqrt{\left( \mathcal{S} \cdot, \cdot \right)}$ is a norm equivalent to the $H^{\frac{1}{2}} \times L^{2}$-norm, we obtain, thanks to estimates \eqref{controlUV}, \eqref{controlnonlin} and \eqref{firstenergyest} and Lemma \ref{controlremainder}, 

\begin{equation*}
\frac{d}{dt}  \mathcal{F}(e) \leq C \left(\frac{1}{\mu_{\min}}, \mu_{\max}, M \right) \left(\beta  \mathcal{F}(e) + \epsilon t \sqrt{\mathcal{F}(e)} \right),
\end{equation*}

\noindent where $\mathcal{F}(e) = \underset{|\alpha| \leq N-4}{\sum} \left(S \partial_{x}^{\alpha} e, \partial_{x}^{\alpha} e \right)$. Then, we get

\begin{equation*}
\mathcal{F}(e)(t) \leq \delta C \left(\frac{1}{\mu_{\min}}, \mu_{\max}, M \right) \int_{0}^{t} \mathcal{F}(e)(s) + s \sqrt{\mathcal{F}(e)(s)} ds.
\end{equation*}

\noindent Denoting $\mathcal{M}(t) = \underset{[0,t]}{\max} \sqrt{\mathcal{F}(e)(t)}$, we have

\begin{equation*}
\mathcal{M}(t) \leq \delta C \left(\frac{1}{\mu_{\min}}, \mu_{\max}, M \right) \int_{0}^{t} \mathcal{M}(s) + s ds,
\end{equation*}

\noindent and the result follows from the Gr\"onwall's lemma.
\end{proof}

\subsection{Global error estimate }

\noindent In this part, we prove our main result. We denote by 

\begin{equation*}
\textbf{U}_{k} =  \left(\mathcal{Y}^{\Delta t}\right)^k  \textbf{U}_{0} 
\end{equation*}

\noindent the \textcolor{black}{approximate} solution and by $ \Ub(t_k) := \Phi^{k \Delta t} \Ub_0$ the exact solution at the time $t_k=k \Delta t$. 

\begin{thm}\label{thm_global_error}
Let \textcolor{black}{$N \geq 7$}, $M > 0$,  $\epsilon,\beta, \mu$ satisfying Condition \eqref{parameters_constraints}  and $b \in L^{\infty}(\R)$. Let \upshape$\textbf{U}_{0} = \left( \zeta_{0}, v_{0} \right)$ \itshape such that,

\upshape
\begin{equation*}
\lver b \rver_{L^{\infty}} + \lver \textbf{U}_{0} \rver_{H^{N+\frac{1}{2}} \times H^{N}} \leq M.
\end{equation*}
\itshape

\noindent Let \upshape$\textbf{U}_{0}(t, \cdot)$ \itshape the solution of the Saut-Xu equations \eqref{saut-xu-deep} with initial data \upshape$\textbf{U}_{0}$ \itshape defined on $\left[0, \frac{T}{\delta} \right]$. Then, there exist constants $\gamma, \nu, \Delta t_0,  C_0>0$ such that for all $\Delta t \in ]0,\Delta t_0]$ and for all $n \in \N$ such that $0 \leq n \Delta t \leq \frac{T}{\delta }$,

\upshape
\begin{equation*}
| \Ub_n|_{H^{N + \frac{1}{2}} \times H^{N}} \leq \nu  \text{  and  }  \lver \Phi^{n \Delta t} \left( \textbf{U}_{0} \right) - \left( \mathcal{Y}^{\Delta t} \right)^{n} \left( \textbf{U}_{0} \right) \rver_{H^{N-4 + \frac{1}{2}} \times H^{N-4}} \leq \gamma \Delta t.
\end{equation*}
\itshape

\end{thm}

\begin{proof}

\noindent The proof is based on a Lady's Windermere's fan argument and is similar to the one in \cite{carles_split} (see also \cite{HoldenLR13}). In order to simplify the notations, we forget the dependence on $\frac{1}{\mu_{\min}}$ and $\mu_{\max}$ in all the constants. We denote by $X^{N}$ the following space

\begin{equation*}
X^{N} = H^{N + \frac{1}{2}} \times H^{N}.
\end{equation*}

\noindent Thanks to Theorem \ref{existence_saut_xu}, there exists $\rho$ such that, for all $\textcolor{black}{t^k = k \Delta t } \in \left[ 0,\frac{T}{\delta } \right]$,

\begin{equation*}
\lver \Ub(t_k) \rver_{X^{N}} = \lver \Phi^{t_k}  \textbf{U}_{0} \rver_{X^{N}} \leq \rho. 
\end{equation*}

\noindent We prove by induction that there exists $\Delta t_0, \gamma, \nu$ such that if $0 < \Delta t \leq \Delta t_0,$ for all $n \in \N$ with $n \Delta t \leq \frac{T}{\delta}$,

\begin{align*}
&(i) \text{  } |\Ub_n-\Ub(t_n) |_{X^{N-4}} \leq \gamma \Delta t,\\
&(ii) \text{  } |\Ub_n|_{X^{N}} \leq e^{C_{3}(M_{1}) \delta n \Delta t} \lver \textbf{U}_{0} \rver_{X^{N}} \leq M_{0},\\
&(iii) \text{  } |\Ub_n|_{X^{N-2}} \leq M_{1},\\
&(iv) \text{  } |\Ub_n|_{X^{N-4}} \leq 2 \rho,\\
\end{align*}

\noindent with 

\begin{align*}
&M_{1} =  e^{C_{3}(2 \rho) T} M, M_{0} = e^{C_{3}(M_{1}) T} M, \gamma = T \max(K,1) C_{4} \left(e^{C_{0}(M_{0}) T} M_{0} \right),\\
&\Delta t_{0} = \min \left(T,T_{0}(M_{0}), T_{3}(M_{0}), T_{4}(M_{0}), \frac{\rho}{\gamma} \right),K=K \left( M_{0} e^{T C_{0}(M_{0})} \right)
\end{align*}

\noindent where $C_{0}$, $T_{0}$, $C_{3}$, $T_{3}$, $C_{4}$, $T_{4}$, $K$ are constants from Theorem \ref{existence_saut_xu}, Remark \ref{rem_op_split}, Proposition \ref{control_localerror} and Inequality \eqref{ConstLip}. The above properties are satisfied for $n=0$. Let $n\geq 1$, and suppose that the induction assumptions are true for $0 \leq k \leq  n-1$. First, we have the following telescopic series (see \cite{HoldenLR13} or \cite{carles_split})

\begin{equation}\label{telecsopic_series}
\textbf{U}_{n} - \Ub(t_n) = \underset{0 \leq k \leq n-1}{\sum}  \Phi^{(n-k-1) \Delta t} \mathcal{Y}^{\Delta t} \textbf{U}_{k} - \Phi^{(n-k-1) \Delta t} \Phi^{\Delta t} \textbf{U}_{k}.
\end{equation}

\noindent For $k\leq n-2$, since $\mathcal{Y}^{\Delta t} \Ub_k = \Ub_{k+1}$, using the induction assumption (ii), we have 

\begin{equation*}
|\mathcal{Y}^{\Delta t} \Ub_k|_{X^{N-3}} \leq M_{0},
\end{equation*}

\noindent and from Theorem \ref{existence_saut_xu}, we get
  
\begin{equation*}
|\Phi^{\Delta t} \Ub_k|_{X^{N-3}}  \leq  e^{C_{0}(M_{0}) T} M_{0}. 
\end{equation*}

\noindent Therefore, from  Proposition \ref{lip_saut_xu} and  up to replacing $K=K \left( M_{0} e^{T C_{0}(M_{0})} \right)$ with $\max(K,1)$, we obtain, for $k \leq n-1$ and $n \Delta t \leq \frac{T}{\delta}$, 

\begin{equation*}
 \lver \Phi^{(n-k-1) \Delta t} \mathcal{Y}^{\Delta t} \textbf{U}_{k} - \Phi^{(n-k-1) \Delta t} \Phi^{\Delta t} \textbf{U}_{k} \rver_{X^{N-4}}  \leq K  \lver \mathcal{Y}^{\Delta t} \Ub_k- \Phi^{\Delta t} \Ub(t_k) \rver_{X^{N-4}}. 
\end{equation*}

\noindent Then, using Proposition \ref{control_localerror} and Inequality (ii), we infer

\begin{equation*}
 \lver \Phi^{(n-k-1) \Delta t} \mathcal{Y}^{\Delta t} \textbf{U}_{k} - \Phi^{(n-k-1) \Delta t} \Phi^{\Delta t} \textbf{U}_{k} \rver_{X^{N-4}}  \leq \delta C_{4} \left( e^{C_{0}(M_{0}) T} M_{0} \right) K (\Delta t)^2. 
\end{equation*}

\noindent Therefore, using the telescopic series \eqref{telecsopic_series}, we get

\begin{equation*}
|\textbf{U}_{n} - \Ub(t_n)|_{X^{N-4}} \leq n C_{4} \left( e^{C_{0}(M_{0}) T} M_{0} \right) K \delta (\Delta t)^2 \leq C_{4} \left( e^{C_{0}(M_{0}) T} M_{0} \right) K T \Delta t.
\end{equation*}

\noindent For Estimate (ii), using Remark \ref{rem_op_split} and the induction assumptions (iii) and (ii), we have,

\begin{equation*}
\lver \textbf{U}_{n} \rver_{X^{N}} = \lver \mathcal{Y}^{\Delta t} \left( \textbf{U}_{n-1} \right) \rver_{X^{N}} \leq e^{\delta C_{3}(M_{1}) \Delta t} \lver \textbf{U}_{n-1} \rver_{X^{N}} \leq e^{C_{3}(M_{1}) \delta n \Delta t} \lver \textbf{U}_{0} \rver_{X^{N}} \leq M_{0}.
\end{equation*}

\noindent We get Estimate (iii) in the same way, using  the induction assumptions (iv) and (iii). Finally, for Estimate (iv), using (i), we have

\begin{equation*}
\lver \textbf{U}_{n} \rver_{X^{N-4}} \leq \lver \textbf{U}_{n} - \textbf{U}(t_{n}) \rver_{X^{N-4}} + \lver \textbf{U}(t_{n}) \rver_{X^{N-4}} \leq \gamma \Delta t  + \rho \leq 2 \rho.
\end{equation*}

\end{proof}

\section{Numerical experiments}\label{numsection}

\noindent The aim of this section is to numerically verify the Lie method convergence rate in $\mathcal{O}(\Delta t)$ for the Saut-Xu system \eqref{saut-xu-deep} and to illustrate some physical phenomena.  

\medskip

\noindent In other works and particularly on the whole water waves problem  (see for example \cite{Craig_Sulem_1}, \cite{numeric_Guyenne_moving_bott}, \cite{nicholls_stab_dn} and references therein), several authors use a discrete Fourier transform even for the transport part. They observe spurious oscillations in the wave profile that lead to instabilities. These errors seem to appear when they evaluate the nonlinear part via Fourier transform because additional terms appear in the approximation, this is the aliasing phenomenon. To fix this problem, they apply at every time step a low-pass filter. The main interest of our scheme is that we do not need one 
\textcolor{black}{because we use a finite difference method to approximate the nonlinear part. }

\medskip

\noindent \textcolor{black}{
For the dispersive equation \eqref{dispersive_part}, we use the forward Euler  discretization in time and for the spatial discretization we consider the Fast Fourier Transform (FFT) implemented in Matlab. In this scheme, 
the interval $[0,1]$ is discretized by $N$ equidistant points, with spacing $\Delta x = 1/N$. The spatial grid points are then given by
$x_j = j/N$, $j=0,...,N$. Therefore, if $u_j(t)$ denotes the approximate
solution to $u(t,x_j)$, the discrete Fourier transform of the sequence
$ \left\lbrace u_j \right\rbrace_{j = 0}^{N-1} $ is defined by 
\begin{equation*}
\hat{u}(k) = \mathcal{F}^d_k (u_j) = \sum_{j=0}^{N-1} u_j e^{-2i \pi j k /N },
\end{equation*}
for $k = 0,\cdots, N-1$, and the inverse discrete Fourier transform is given by
\begin{equation*}
u_j = \mathcal{F}_{j}^{-d}( \hat{u}_k ) = \frac{1}{N} \sum_{k = 0}^{N -1} \hat{u}_k e^{ 2 i \pi k x_j}, 
\end{equation*}
for $j = 0, \cdots,N-1 $. Here $\mathcal{F}^d$ denotes the discrete
Fourier transform and $\mathcal{F}^{-d}$ its inverse.}

\color{black}
\noindent Then, in what follows the numerical scheme to solve \eqref{dispersive_part} is given by 
\begin{equation}
\begin{pmatrix} \zeta^{n+1}_j \\ v_j^{n+1} \end{pmatrix} = \begin{pmatrix} \zeta^{n}_j \\ v_j^{n} \end{pmatrix} - \Delta t \begin{pmatrix} F_j^{n} + S_j^n \\ G_j^{n} \end{pmatrix}  
\label{subeq1}
\end{equation}

where $S_j^n = \beta \sqrt{\mu} \mathcal{F}^{-d}_j(i k \mathcal{F}^d_k (B_\mu v^n_j ))$ and $F_j^n = I_1 + I_2 $ with 

\begin{align*}
I_1 =  \mathcal{F}^{-d}_j \Big(  i \, \tanh(\sqrt{\mu} k) \Big( -1 &+ \frac{\epsilon \sqrt{\mu}}{2}  \mathcal{F}^d_k( v^n_j \mathcal{F}^{-d}_j(k \tanh(\sqrt{\mu k} ) \hat{\zeta}^n_k ) )  \\
&+ \epsilon \sqrt{\mu}   \mathcal{F}^d_k( \zeta^n_j \mathcal{F}^{-d}_j(k \tanh(\sqrt{\mu k} ) \hat{v}^n_k ) )    \Big)   \Big) 
\end{align*}

\begin{equation*}
I_2 = \zeta_j^n \mathcal{F}^{-d}_j\left( i k \hat{v}^n_k \right) + \frac{1}{2} \mathcal{F}^{-d} \left( i k \hat{\zeta}^n_k  \right) \mathcal{F}_j^{-d} \left( \tanh(\sqrt{\mu} k )^2 \hat{v}^n_k \right).
\end{equation*}


\medskip

\noindent To approximate the equation \eqref{transport_part}, we use the following finite difference scheme
\begin{equation}
\begin{pmatrix} \zeta^{n+1}_j \\ v_j^{n+1} \end{pmatrix} =  \begin{pmatrix} \zeta^{n}_j \\ v_j^{n} \end{pmatrix}  - \Delta t  \frac{\epsilon \sqrt{\mu} }{2}  \begin{pmatrix}   G^n_1  \\ 3 G^n_2 \end{pmatrix} 
\label{subeq2}
\end{equation}

where 

\begin{equation*}
G_1 = w_j^n  \frac{\zeta^n_{j+1}-\zeta^n_{j-1}}{2 \Delta x}  - \frac{\Delta t}{2 \Delta x^2} (w^n_j)^2 \left(   \zeta^n_{j-1}-2\zeta^n_j+\zeta^n_{j+1}  \right)
\end{equation*}

\medskip

with $w_j^n = - \mathcal{F}^{-d}_j \left( \tanh( \sqrt{\mu} k)^2 \hat{v}^n_k   \right) + v_j^n$ 
and 
$$G_2^n =  \frac{(v^n_{j+1})^2 - (v^n_{j-1})^2 }{2 \Delta x}  - \frac{\Delta t}{2 \Delta x^2}  \left(  v^n_{j+1/2} \left( (v^n_{j+1})^2 - (v^n_{j})^2 \right) - v^n_{j-1/2} \left( (v^n_j)^2 - (v^n_{j-1} )^2  \right) \right)  $$ 
with $v^n_{j \pm 1/2} = \frac{ v^n_j + v^n_{j \pm 1} }{2}. $

\medskip

\noindent We remarked  that for our numerical simulations, it is not necessary to decompose the term $v \partial_x \zeta$  (see Remark \ref{Tdecomposition}) to get the numerical convergence. Indeed, it seems that since the time step is chosen very small, we obtain a solution of the dispersive equation for each iteration. In this case, we do not need to evaluate the term $
\partial_{x} \zeta \Hmu^{2} v$.
\medskip

\noindent To ensure the validity of our numerical simulations, we have to be careful of the numerical instability, that why the time and the space steps are chosen in a way that the following CFL condition is satisfied: 
 \begin{equation}
 \label{CFL}
 |v | \frac{\Delta t}{\Delta x} < 1. 
 \end{equation}

\color{black}

\medskip
\medskip


\subsection{Example 1: Convergence curve} 

\noindent In this example, we consider the following initial data: 

\begin{equation*}
\zeta_0(x) = \sech   \left(\frac{\sqrt{3}}{2} x \right), \quad v_0= \zeta_0.
\end{equation*}

\noindent with two different bathymetries: a bump \textcolor{black}{ ( $b(x) = \cos(x)$ ) } and a ripple  bottom
\textcolor{black}{
$$  \hspace{0.5cm} b(x) = \left\{
\begin{array}{l}
  0.5 - \frac{1}{18} (x-8)^2 \mbox{ if }  5 \leq x \leq 11 \\
  0 \mbox{ otherwise }. 
\end{array}
\right. $$
}

 Note that in order to avoid numerical reflections due to the boundaries and justify of the use of the Fast Fourier Transform, we  decide to take rapidly decreasing initial data. Figures \ref{evolcos}  and \ref{evolbosse}  display the evolution for different times of the free surface $\zeta$ for these two test cases. \textcolor{black}{We decided to take $\epsilon = 0.1, \mu = 1, \beta = \frac{1}{2}$, $N = 2^8, \Delta x = 2L/N, T = 10$ where N is the mesh modes number, $L=30$ the length of the domain and $T$ the final time. Note that the time step $\Delta t $ is chosen iteratively in a way that the CFL condition \eqref{CFL} is satisfied.  }

\begin{figure}[!t]
\centering
\includegraphics[scale=0.4]{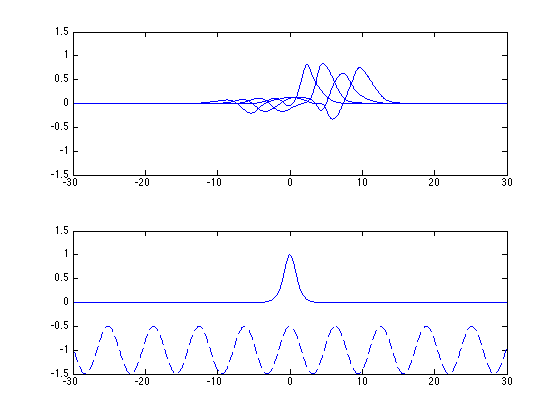}
\caption{Upper: Evolution of the free surface for different times \textcolor{black}{t=2.5,5,7.5 and 10}. 
Lower: bottom topography and initial condition.}
\label{evolcos}
\end{figure}

\begin{figure}[!t]
\centering
\includegraphics[scale=0.4]{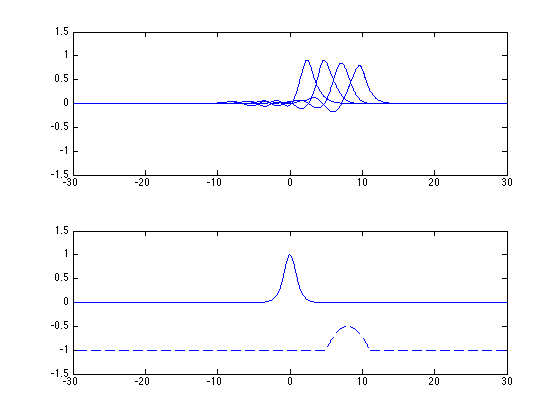}
\caption{Upper: Evolution of the free surface for different times  \textcolor{black}{t=2.5,5,7,5 and 10 }. Lower: bottom topography and initial condition.}
\label{evolbosse}
\end{figure}

\medskip

\noindent Figures \ref{figconv} displays the convergence curve for this example. We plot the logarithm of the  error (in norm $H^1 \times L^2$) in function of the logarithm of the time step $\Delta t.$ The convergence numerical order is then given by the slope of this curve. For reference, a small line (the dashed line) of slope one is added in this figure. We see that the numerical rate of convergence is greater than 1.

\medskip

\begin{figure}[!t]
\centering
\includegraphics[scale=0.25]{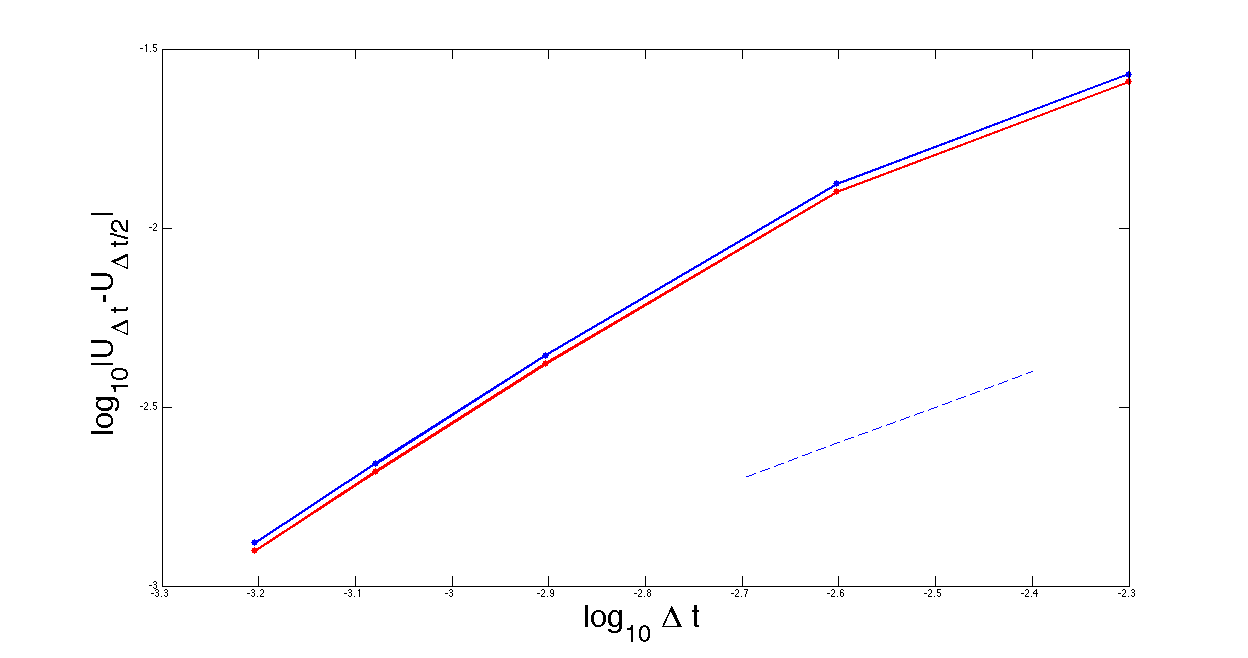}
\caption{Convergence curve (for the $H^1 \times L^2$-norm) for the Lie method for two bottoms : bump (black line) and ripple bottom (blue line) \textcolor{black}{for $T=10.$} }
\label{figconv}
\end{figure}

\color{black}
\subsection{Example 2: Non smooth topographies} 

\noindent In this example we study the evolution of water waves over a rough bottom. This problem is still a mathematical issue. Many models derived from the Euler equations suppose that the bathymetry is smooth. Even worse, a non smooth bathymetry introduces singular terms in these models. This issue is particularly easy to see for shallow waters models. To handle this, Hamilton (\cite{hamilton_rapidlybott}) and Nachbin (\cite{nachbin_nonsmoooth}) used a coformal mapping to derive long waves models. Notice also the work of Cathala (\cite{Cathala}) who derived alternatives Saint-Venant equations and Boussinesq systems with non smooth topographies which do not involve any singular terms. We notice that the Saut-Xu equations \eqref{saut-xu-deep} can handle a non smooth topography (see Theorem \eqref{existence_saut_xu}) and our numerical scheme too (see Theorem \eqref{thm_global_error}).

\medskip

\noindent In the following, we give an example with a non smooth bathymetry. We consider the following initial conditions and bathymetry  

\begin{equation*}
\zeta_{0}(x) = v_{0}(x) = e^{-x^{2}} \text{   and   } b(x)=\frac{\beta}{4}(1+\tanh(100(x-2)))(1-\tanh(100(x-8))).
\end{equation*}

\noindent We decided to take $\mu =1$, $\beta=0.5$, $\epsilon=0.1$, $N = 2^8, \Delta x = 2L/N$ where N is the mesh modes number, $L=20$ the length of the domain. The time step $\Delta t $ is chosen iteratively in a way that the CFL condition \eqref{CFL} is satisfied. Figure \ref{figBottNonSmoo} displays the evolution of the surface $\zeta$ for different times over the bottom.

\begin{figure}[!t]
\centering
\includegraphics[scale=0.3]{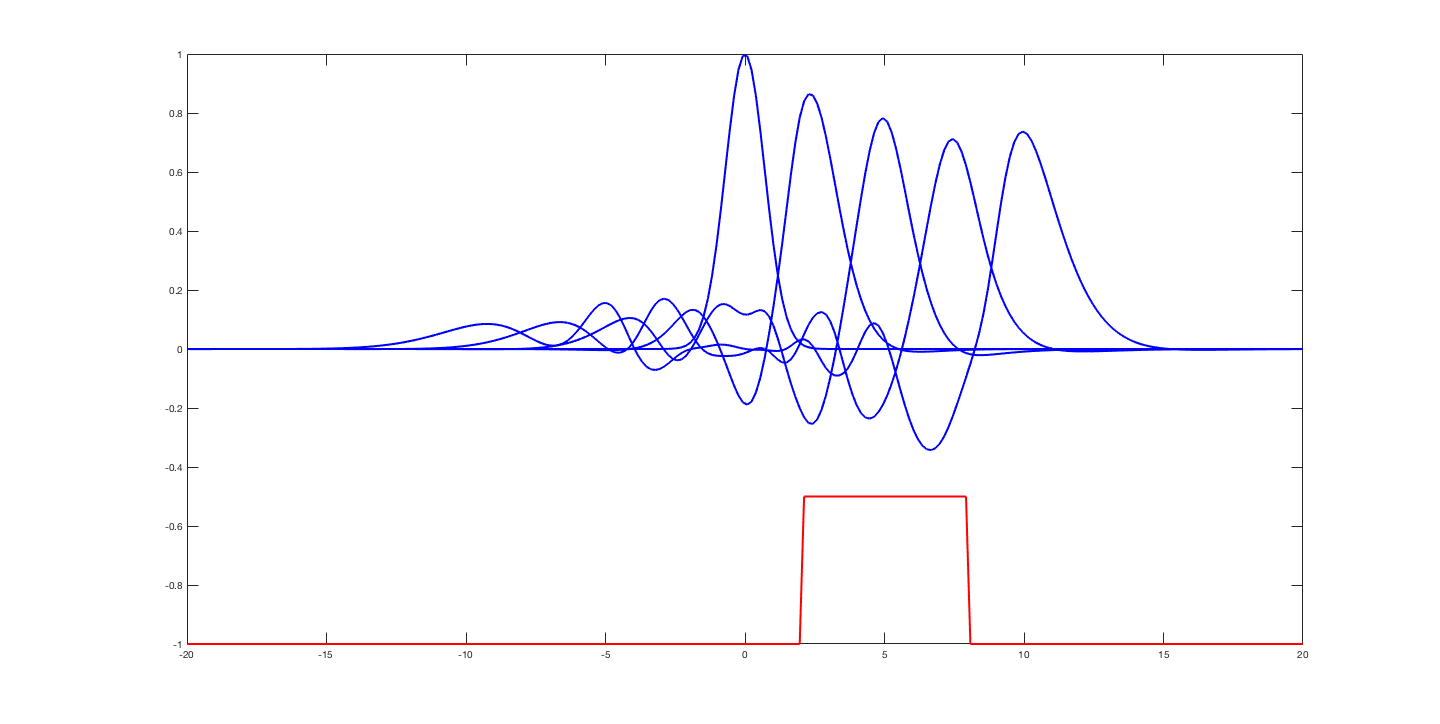}
\caption{Evolution of the free surface (blue lines) for different times t=0,3,6,9 and 12 over a rough bottom (black line).}
\label{figBottNonSmoo}
\end{figure}

\color{black}

\subsection{Example 3: Boussinesq regime}

\noindent \color{black} In Section \ref{proof_splitting}, we crucially use the fact that $\mu$ is bounded from below. In this example, we test our scheme for small values of $\mu$ (also called the shallow water regime). We show that our scheme is still valid even if we do not have a proof of the convergence of our scheme in this regime. \color{black} In the shallow water regime, there is a huge literature for asymptotic models (see for instance \cite{Lannes_ww}). Among all these asymptotic models, we have the KdV equation. It is a model obtained under the Boussinesq regime, i.e. when $\epsilon = \mu$, $\beta = 0$ and $\mu$ small. In the following, we formally derive a KdV equation from the Saut-Xu equations and we give numerical simulations in this setting.

\medskip
\medskip

\noindent We recall that, without the assumption $\nu = \frac{1}{\mu}$, the Saut-Xu equations are given by the system \eqref{saut-xu}. Notice also that

\begin{equation}\label{dt_Hmu}
\mathcal{H}_{\mu} = - \sqrt{\mu} \partial_{x} - \frac{1}{3} \mu^{\frac{3}{2}} \partial_{x}^{3} + \mathcal{O}(\mu^{2}).
\end{equation}

\noindent Then if we assume that $\mu = \epsilon$, $\nu=1$ (since $\nu \sim 1$ if $\mu$ is small) and we drop all the terms of order $\mathcal{O}(\mu^{2})$ in System \eqref{saut-xu}, we obtain the following equations

\begin{equation}\label{saut-xu_boussi}
\left\{
\begin{aligned}
&\partial_{t} \zeta + \partial_{x} v + \mu v \partial_{x} \zeta - \frac{1}{2} \mu^{\frac{3}{2}} v \partial_{x} \zeta + \frac{1}{3} \mu \partial_{x}^{3} v + \mu \zeta \partial_{x} v  =0,\\
&\partial_{t} v + \partial_{x} \zeta + \mu v \partial_{x} v + \mu^{\frac{3}{2}} \frac{1}{2} v \partial_{x} v= 0.
\end{aligned}
\right.
\end{equation}

\noindent Formally, the solutions of this system are close to the solutions of \eqref{saut-xu} with an accuracy of order $\mathcal{O}(\mu^{2})$. Notice that this system is not a standard Boussinesq system (in the sense of \cite{bona_chen_saut_derivation} or \cite{Lannes_ww}) because of our nonlinear change of variables \eqref{change_variable}. Using the approach developed in \cite{schneider_wayne_longwave}, \cite{bona_colin_lannes}, \cite{Alvarez_Lannes} (see also Part 7.1.1 in \cite{Lannes_ww}) we can check that, formally, the following KdV equation is an asymptotic model of the system \eqref{saut-xu_boussi}

\begin{equation}\label{kdv_eq}
\partial_{t} f + \frac{3}{2} f \partial_{x} f + \frac{1}{6} \partial_{x}^{3} f = 0.
\end{equation}

\noindent This means that if we solve \eqref{saut-xu_boussi} with the initial data $\left(f_{0}, f_{0} \right)$ and \eqref{kdv_eq} with the initial datum $f_{0}$, the solution $\left(\zeta, v \right)(t,x)$ of \eqref{saut-xu_boussi} is close to $\left(f,f \right) (\mu t,x-t)$. Furthermore, if we take $f_{0}(x) = \alpha \text{sech}^{2} \left(\sqrt{\frac{3}{4} \alpha} x \right)$, the solution $f$ of the KdV equation with this initial datum is the soliton $f(t,x) = f_{0}(x-ct)$ with $c=\frac{\alpha}{2}$. Hence, in this case, the solution of \eqref{saut-xu_boussi} and \eqref{saut-xu} are close to a soliton.

\medskip
\medskip

\noindent  In the following we check that the solution to \eqref{saut-xu-deep} is indeed close to the KdV solution when $\mu$ is small. We simulate one soliton. We took \color{black}{$v_0(x) = \zeta_0(x) = \sech^2  \left(\frac{\sqrt{3}}{2} x \right)$ \color{black}, $\epsilon = \mu = 0.01$, $\alpha = 1$ and the final time is $T=10$. \textcolor{black}{We decided to take $N = 2^9, \Delta x = 2L/N$ where N is the mesh modes number, $L=30$ the length of the domain. The time step $\Delta t $ is chosen iteratively in a way that the CFL condition \eqref{CFL} is satisfied.}  Figure \ref{evolsoliton} represents the evolution of this soliton at different times. 
\color{black} Hence, our sheme is still valid when $\mu$ is small. \color{black}
\medskip
\medskip

\noindent In deep water ($\mu$ not small), the KdV approximation ceases to be a good approximation. In order to get some insight on the range of validity of the KdV approximation, we compare in Figure \ref{variationmu} the solution of \eqref{saut-xu-deep} to the exact soliton after a time $T=10$ for various values of $\mu$. \textcolor{black}{We took the same numerical parameters that before.} We notice that even for $\epsilon = \mu=0.1$ and a final time $T=\frac{1}{\mu}$, the KdV approximation remains a good approximation of the Saut-Xu system.

\begin{figure}[!t]
\centering
\includegraphics[scale=0.2]{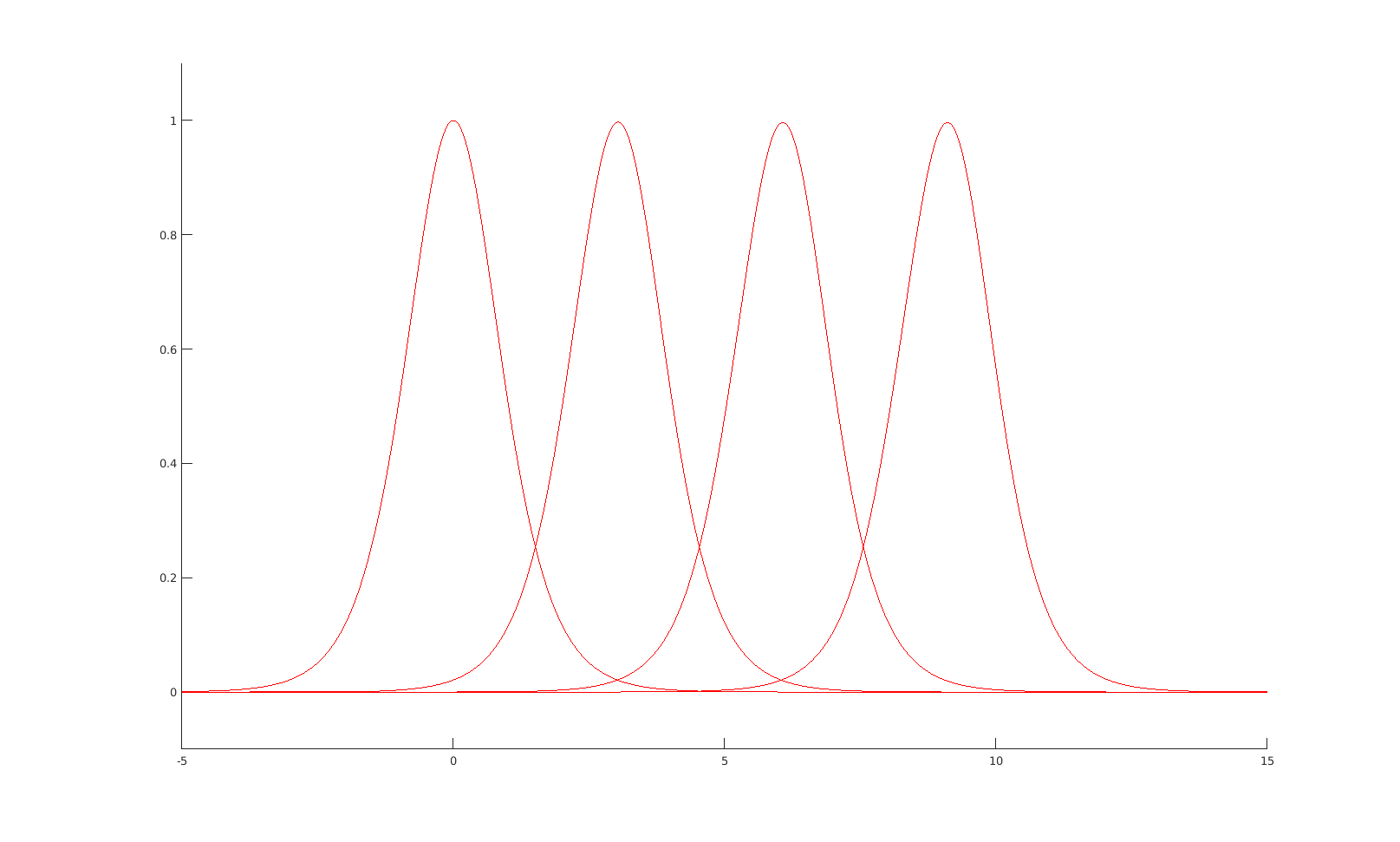}
\caption{Evolution of the soliton at different times $t=0,3,6,9$ ($\epsilon = 0.01$).}
\label{evolsoliton}
\end{figure}

\begin{figure}[!t]
\centering
\includegraphics[scale=0.25]{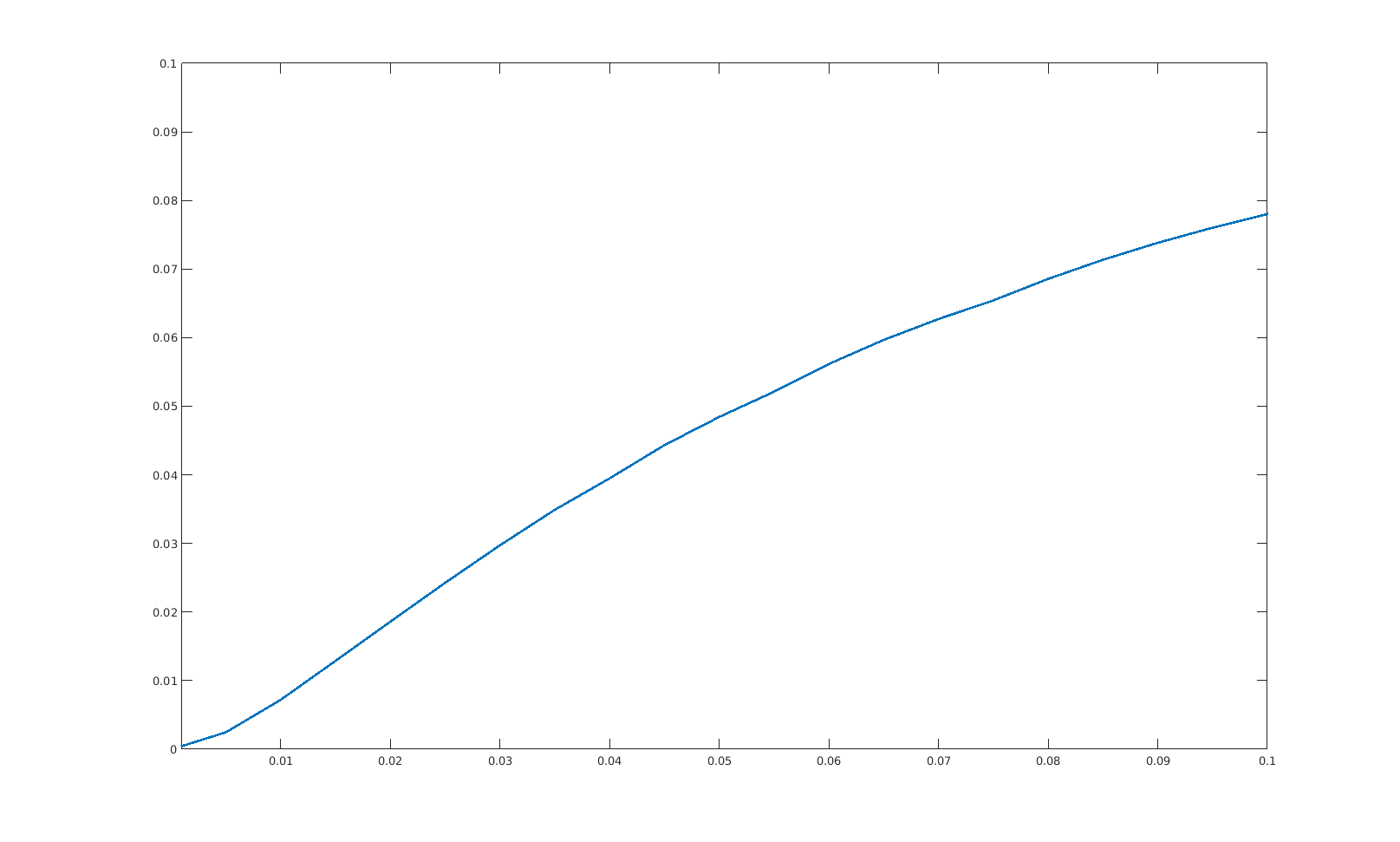}
\caption{Difference after a time $T=10$ between a real soliton and a soliton generated by our scheme with the same initial data and for different values of $\epsilon=\mu$. Abscissa : value of $\epsilon$; Ordinate : \color{black} quotient of the difference after a final time $T=10$ by the maximum of the soliton. \color{black}}
\label{variationmu}
\end{figure}

\subsection{Example 4: Rapidly varying topographies}

\noindent In this example we study the evolution of water waves over a rapidly varying periodic bottom. We assume that $\mu =1$. This problem is linked to the Bragg reflection phenomenon (see for instance \cite{Mei_bragg}, \cite{liu_bragg}, \cite{numeric_Guyenne_moving_bott}). We take

\begin{equation}\label{ini_data_rapidly_bot}
\zeta_{0} = v_{0} = \text{sech}^{2}\left(\frac{\sqrt{3}}{2} x \right) \text{   and   } b(x)=\cos(\alpha x).
\end{equation}

\noindent\textcolor{black}{We decided to take $N = 2^9, \Delta x = 2L/N$ where N is the mesh modes number, $L=30$ the length of the domain. The time step $\Delta t $ is chosen iteratively in a way that the CFL condition \eqref{CFL} is satisfied.} Figure \ref{evolfastbot} compares the evolution of water waves when we take $\alpha=10$ (blue line) and when we take $b(x)=0$ (blue line). Figure \ref{difffastbott} displays the difference between the case of a flat bottom and the case of a bottom of the form $b(x)=\cos(\alpha x)$ for different values of $\alpha$. We observe an homogenization effect when $\alpha$ is large. It seems that when $\alpha$ goes to infinity, a solution of the Saut-Xu equations converges to a solution of the Saut-Xu equations with a flat bottom (corresponding to the mean of $b$). Notice that this result is different from what we could see in the literature ( for instance \cite{chupin_neumann} or \cite{craig_lannes_sulem}), since we take a bottom of the form $b(x)=\cos
 (\alpha x)$ and not of the form $b(x)= \frac{1}{\alpha} \cos(\alpha x)$. Our numerical simulations suggest therefore a homogenization effect for large amplitude bottom variations that has not been investigated so far.

\begin{figure}[!t]
\centering
\includegraphics[scale=0.25]{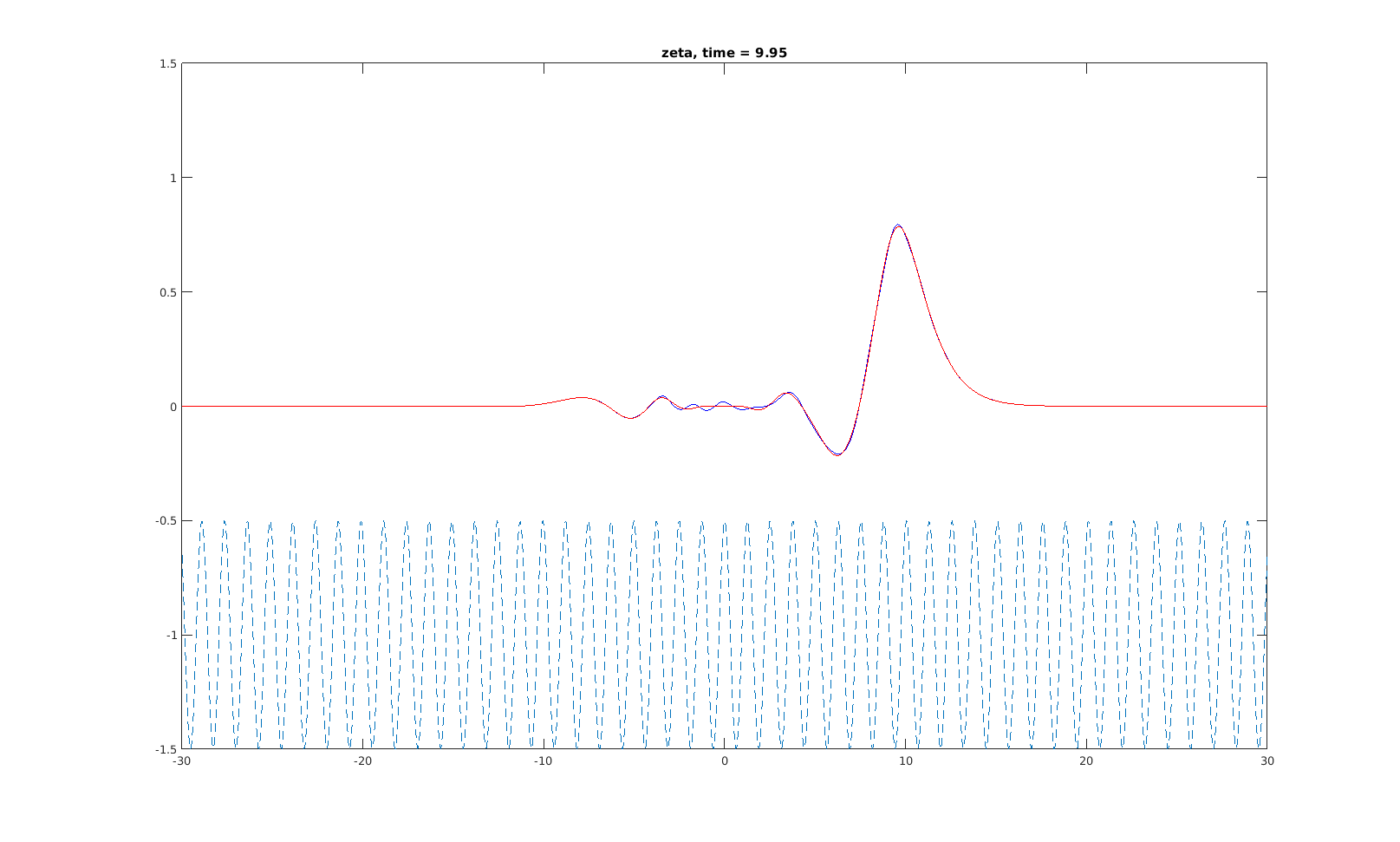}
\caption{Comparison between the evolution of a water wave (blue line) over a bottom of the form $b(x)=\cos(10x)$ (dashed line) and the evolution of a water wave over a flat bottom (black line). $\epsilon =0.05$, $\beta=0.5$.}
\label{evolfastbot}
\end{figure}

\begin{figure}[!t]
\centering
\includegraphics[scale=0.4]{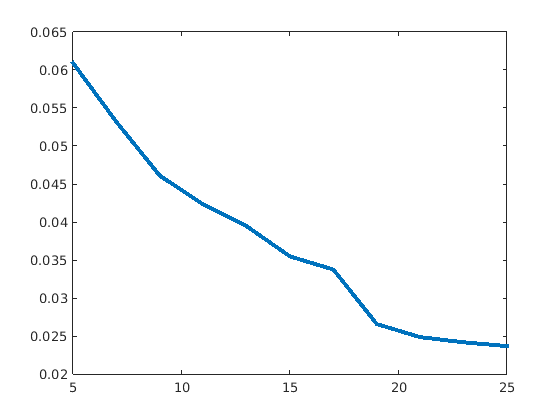}
\caption{Difference between a water wave over a rapidly varying topography $b(x)=\cos(\alpha x)$ and a water wave over a flat bottom. Abscissa : value of $\alpha$; Ordinate : \textcolor{black}{quotient of the difference after a final time $T=10$ by the maximum of $\zeta_{0}$.}}
\label{difffastbott}
\end{figure}

\newpage
\appendix 
\section{}

\noindent In this part, we give some estimate for the operator $\Hmu$ and some standard product and commutator estimates. For the estimates for $\Hmu$, we refer to part III in \cite{saut_xu}. For the other estimates we refer to \cite{alinhac_gerard} and \cite{Lannes_sharp_estimates}. We recall that $\Hmu$ is defined by

\begin{equation*}
\Hmu = - \frac{\tanh(\sqrt{\mu} D)}{D} \partial_{x}.
\end{equation*}

\noindent First, we show that $\Hmu$ is a zero-order operator.

\begin{prop}\label{control_Hmu}
\noindent Let $s \geq 0$ and $\mu$ satisfying Condition \eqref{parameters_constraints}. Then,

\begin{equation*}
\lver \Hmu u \rver_{H^{s}}  \leq  C \left(\mu_{\max} \right) \lver u \rver_{H^{s}}.
\end{equation*}

\noindent Furthermore, for all $s \geq r \geq 0$,

\begin{equation*}
\lver \left( \Hmu^{2} + 1 \right) u \rver_{H^{s}}  \leq  C_{r} \left(\frac{1}{\mu_{\min}} \right) \lver u \rver_{H^{r}}.
\end{equation*}

\end{prop}

\noindent Then, we give a commutator estimate for $\Hmu$. 

\begin{prop}\label{Hmu_commut}
\noindent Let $s \geq 0$, $t_{0} > \frac{1}{2}$, $r \geq 0$, and $\mu$ satisfying Condition \eqref{parameters_constraints}. Then,

\begin{equation*}
\lver \left[\Hmu, a \right] u \rver_{2} \leq C \lver a \rver_{H^{t_{0}}} \lver f \rver_{2}, 
\end{equation*}

\begin{equation*}
\lver |\xi|^{s} \widehat{\left[\Hmu, a \right] u} \rver_{2} \leq C \left(\frac{1}{\mu_{\min}} \right) \lver a \rver_{H^{r+s}} \lver \frac{\left(1+|\xi| \right)^{t_{0}}}{|\xi|^{r}} \widehat{u} \rver_{2},
\end{equation*}

\noindent and

\begin{equation*}
\lver |\xi|^{s} \widehat{\left[\Hmu, a \right] u} \rver_{2} \leq C \left(\frac{1}{\mu_{\min}} \right) \lver a \rver_{H^{r+s+t_{0}}} \lver \frac{1}{|\xi|^{r}} \widehat{u} \rver_{2}.
\end{equation*}

\end{prop}

\noindent We recall the well-known Coifman-Meyer estimate. We recall also that $\Lambda$ is the Fourier multiplier $\Lambda = \sqrt{1+D^{2}}$.

\begin{prop}\label{coifmanmeyer}
\noindent Let $s > \frac{3}{2}$, $u \in H^{s}(\R)$ and $v \in H^{s-1}(\R)$. Then we have the following commutator estimate

\begin{equation*}
\lver \left[ \Lambda^{s}, u \right] v \rver_{2} \leq C \lver u \rver_{H^{s}} \lver v \rver_{H^{s-1}}.
\end{equation*} 

\end{prop}

\noindent We recall also the following product estimate.

\begin{prop}\label{productestim}
\noindent Let $s_{1}, s_{2},s$ such that $s_{1} + s_{2} \geq 0$, $s \leq \min \left( s_{1}, s_{2} \right)$ and $s < s_{1} + s_{2} - \frac{1}{2}$. Let $u \in H^{s_{1}}(\R)$ and $v \in H^{s_{2}}(\R)$. Then,

\begin{equation*}
\lver u v \rver_{H^{s}} \leq C \lver u \rver_{H^{s_{1}}} \lver v \rver_{H^{s_{2}}}.
\end{equation*}
\end{prop}

\section*{Acknowledgments}

\noindent The second author has been partially funded by the ANR project Dyficolti ANR-13-BS01-0003.

\footnotesize
\bibliographystyle{plain}
\bibliography{biblio}
\normalsize

\end{document}